\documentclass[a4paper]{article}

\usepackage{a4wide}
\usepackage{amsmath}
\usepackage{amssymb} 
\usepackage{bm}
\usepackage{booktabs}
\usepackage{algorithm}
\usepackage{multicol}
\usepackage{multirow}
\usepackage[T1]{fontenc}
\usepackage{lmodern, textcomp}
\usepackage{amsfonts}
\usepackage{graphicx}
\usepackage{epstopdf}
\usepackage{algorithmic}
\usepackage{cleveref}
\usepackage{amsthm}
\usepackage{url}
\usepackage{color}
\usepackage{subcaption}

\newtheorem{theorem}{Theorem}[section]
\newtheorem{lemma}{Lemma}[section]

\title{Block cross-interactive residual smoothing for\\ Lanczos-type solvers for linear systems with\\ multiple right-hand sides\thanks{\textbf{Funding:} This study was supported by grant numbers JP20K14356, JP21K11925, 23K21673, 23H00462, JP24K14985 and JP24K00535 from Grants-in-Aid for Scientific Research Program (KAKENHI) of the Japan Society for the Promotion of Science (JSPS).}}

\date{}

\author{Kensuke Aihara\thanks{Department of Computer Science, Faculty of Information Technology, Tokyo City University, 1-28-1 Tamazutsumi, Setagaya-ku, Tokyo 158-8557, Japan 
  ({\tt aiharak@tcu.ac.jp}).}
\and Akira Imakura\thanks{Institute of Systems and Information Engineering, University of Tsukuba, 1-1-1 Tennodai, Tsukuba, Ibaraki 305-8573, Japan 
  ({\tt imakura@cs.tsukuba.ac.jp}, {\tt morikuni@cs.tsukuba.ac.jp}).}
\and Keiichi Morikuni\footnotemark[3]}

\begin{document}

\maketitle

\begin{abstract}
Lanczos-type solvers for large sparse linear systems often exhibit large oscillations in the residual norms. 
In finite precision arithmetic, large oscillations increase the residual gap (the difference between the recursively updated residual and the explicitly computed residual) and a loss of attainable accuracy of the approximations. 
This issue is addressed using cross-interactive residual smoothing (CIRS). 
This approach improves convergence behavior and reduces the residual gap. 
Similar to how the standard Lanczos-type solvers have been extended to global and block versions for solving systems with multiple right-hand sides, CIRS can also be extended to these versions. 
While we have developed a global CIRS scheme (Gl-CIRS) in our previous study [K. Aihara, A. Imakura, and K. Morikuni, SIAM J. Matrix Anal. Appl., 43 (2022), pp.1308--1330], in this study, we propose a block version (Bl-CIRS). 
Subsequently, we demonstrate the effectiveness of Bl-CIRS from various perspectives, such as theoretical insights into the convergence behaviors of the residual and approximation norms, numerical experiments on model problems, and a detailed rounding error analysis for the residual gap. 
For Bl-CIRS, orthonormalizing the columns of direction matrices is crucial in effectively reducing the residual gap. 
This analysis also complements our previous study and evaluates the residual gap of the block Lanczos-type solvers. 
\end{abstract}

\


\


\section{Introduction}

In this study, we consider solving large sparse linear systems with multiple right-hand sides
\begin{align}
	AX = B,\quad A \in \mathbb{R}^{n \times n},\quad B = [ \bm{b}_1, \bm{b}_2,\dots, \bm{b}_s ] \in \mathbb{R}^{n \times s}, \label{eq:AX=B}
\end{align}
where $A$ is nonsingular and not necessarily symmetric, and the number of right-hand sizes $s$ is moderate, that is, $s \ll n$. 
This study also focuses on block Lanczos-type solvers for the problem~\eqref{eq:AX=B} and proposes a refined residual smoothing technique for improving their convergence behavior and attainable accuracy of approximations. 
For a comprehensive discussion, we first outline the background, motivation, and related results of this study.

\subsection{Application problems} 

The problem \eqref{eq:AX=B} is extensively studied in various scientific computing fields; for instance, numerical simulations of electromagnetic scattering and wave propagation \cite{Miller1991, SmithPetersonMittra1989IEEE, SmithPetersonMittra1990IEEE, Smith1987th}, quantum scattering~\cite{YangMiller1989}, structural mechanic~\cite{GuennouniJbilouSadok2003ETNA, PapadrakakisSmerou1990}, chemistry~\cite{ChatfieldReevesTruhlarDuneczkySchwenke1992}, multiple radiation and scattering problems in structural acoustics~\cite{MalhotraFreundPinsky1997}, and lattice quantum chromodynamics~\cite{SakuraiTadanoKuramashi2010, TadanoKuramashiSakurai2010}. 
The problem~\eqref{eq:AX=B} also arises as a subproblem in numerical methods such as the block Lanczos procedure for eigenvalue problems~\cite{Cullum1978BIT,GolubUnderwood1977}, a block variant of the Jacobi--Davidson method~\cite{SleijpenBootenFokkemaVanderVorst1996BIT}, block eliminations by bordering a square matrix with a few rows and columns~\cite{FreundMalhotra1997LAA}, and other projection-type eigensolvers~\cite{IkegamiSakurai2010, IkegamiSakuraiNagashima2010JCAM, Polizzi2009}. 
These challenges in applications have motivated the community to deeply explore fast and accurate numerical algorithms for solving such problems \eqref{eq:AX=B}.

\subsection{Krylov subspace methods for multiple right-hand sides}

This study focuses on using Krylov subspace methods to solve \eqref{eq:AX=B}. 
For a single linear system $A\bm{x}=\bm{b}$ ($\bm{b} \in \mathbb{R}^n$), the standard Krylov subspace is defined as follows: 
\begin{align}
\mathcal K_k(A, \bm{r}_0) := \text{span}(\bm{r}_0, A\bm{r}_0, \dots, A^{k-1}\bm{r}_0), \label{KS}
\end{align}
where $\bm{r}_0:=\bm{b}-A\bm{x}_0$ is the initial residual for a given initial guess $\bm{x}_0 \in \mathbb{R}^n$. 
The Krylov subspace has been extended in two ways to simultaneously solve the $s$ linear systems in \eqref{eq:AX=B}; global and block Krylov subspaces. 
The global- and block-type solvers have advantages over the standard ones. 
In the following, let $X_0 \in \mathbb{R}^{n\times s}$ be an initial guess and $R_0 := B - A X_0$ be the corresponding initial residual for \eqref{eq:AX=B}. 

The global-type solvers generate approximations by constructing the following search subspace (often referred to as a matrix Krylov subspace \cite{JbilouMessaoudiSadok1999APNUM}):
\begin{align}
\mathcal K_k^G(A, R_0) := \left\{ \sum_{i=0}^{k-1} c_i A^i R_0 \, \middle| \, c_i \in \mathbb{R}\right\}, \label{Gl-KS}
\end{align}
which is a simple extension of \eqref{KS} for multiple right-hand sides. 
They can also handle more general linear matrix equations easily and thus have favorable usability. 
Specific global methods (e.g., see \cite{JbilouSadokTinzefte2005ETNA, ZhangDaiZhao2010AMC}) for \eqref{eq:AX=B} correspond to their standard counterparts for $(I_{ns}\otimes A)\hat{\bm{x}} = \hat{\bm{b}}$, where $I_{ns}$ is the identity matrix of order $ns$, $\otimes$ denotes the Kronecker product, and $\hat{\bm{b}} := [\bm{b}_1^\top,\bm{b}_2^\top,\dots,\bm{b}_s^\top]^\top$. 
Therefore, the global-type solvers have essentially the same numerical properties as the corresponding standard ones for a single linear system. 

The block-type solvers utilize a so-called block Krylov subspace
\begin{align}
\mathcal{K}_k^\square (A, R_0) := \left\{ \sum_{i=0}^{k-1} A^i R_0 c_i^\square \, \middle| \, c_i^\square \in \mathbb{R}^{s \times s} \right\}, \label{Bl-KS}
\end{align}
and a column of the $k$th approximation $X_k$ is determined with expanding $s$-dimensional subspace at each iteration. 
Therefore, this type of method determines approximations more efficiently than the standard and global ones. 
Note that \eqref{Bl-KS} simplifies to \eqref{Gl-KS} when $c_i^\square := c_i I_s$ for $c_i \in \mathbb{R}$. 

As in the standard Krylov subspace methods, most of the global- and block-type solvers are broadly classified into the methods based on the (global or block) Lanczos or Arnoldi process. 
In the next subsection, we briefly review the specific block Lanczos-type solvers for \eqref{eq:AX=B} with $s > 1$.

\subsection{Block Lanczos-type solvers}

The most basic block-type solver is the block conjugate gradient (Bl-CG) method for symmetric positive definite linear systems (e.g., see \cite{ChanNg1999SISC, ChanWan1997SISC, GolubRuizTouhami2007SIMAX, NikishinYeremin1995SIMAX, OLeary1980LAA, Saad1987MathComp}). 
Many block-type solvers for nonsymmetric (or non-Hermitian) linear systems have been developed to extend the capabilities of Bl-CG. 
These include block versions of BiCG and its associated product-type methods. 

The underlying block BiCG (Bl-BiCG) method was presented by O'Leary~\cite{OLeary1980LAA}, and its quasi-minimal residual (QMR) variants were developed by Freund and Malhotra~\cite{FreundMalhotra1997LAA} and Simoncini~\cite{Simoncini1997SIMAX}. 
The block BiCG stabilized (Bl-BiCGSTAB) method, which is a representative block product-type BiCG, was proposed by El Guennouni et al.~\cite{GuennouniJbilouSadok2003ETNA}. 
The block Lanczos/Orthodir and BIODIR methods~\cite{GuennouniJbilouSadok2004APNUM} were also developed by the same authors. 
Afterward, several modifications of the block-type solvers were studied; for instance, Tadano et al.~\cite{TadanoSakuraiKuramashi2009} reformulated Bl-BiCGSTAB as the block BiCG gap-reducing (Bl-BiCGGR) method, and Rashedi et al.~\cite{RashediEbadiBirkFrommer2016JCAM} modified the block (and global) versions of BiCG, QMR, and BiCGSTAB to improve their numerical stability. 
Other block product-type BiCG, such as the block BiCGstab($\ell$)~\cite{SaitoTadanoImakura2014JSIAML} and block generalized product-type BiCG (Bl-GPBiCG)~\cite{KuramotoTadano2020TJSIAM, TaherianToutounian2021NUMALGO} methods, were also considered. 
Moreover, Du et al.~\cite{DuSogabeYuYamamotoZhang2011JCAM} developed a block version of the induced dimension reduction (IDR)($s$) method, which extends Bl-BiCGSTAB, and Naito et al.~\cite{NaitoTadanoSakurai2012JSIAMLett} proposed its modified variant. 

We also remark on the numerical stability of block-type solvers. 
Unlike the standard- and global-type solvers, the block-type solvers require inversions of $s\times s$ matrices. 
These matrices can be ill-conditioned as the iteration proceeds, leading to a deterioration in the convergence. 
To address this issue, we orthonormalize the columns of $n \times s$ iteration matrices. 
The effectiveness of the orthonormalization strategy has been demonstrated in numerous studies (e.g., see \cite{Dubrulle2001ETNA, KuramotoTadano2020TJSIAM, NakamuraIshikawaKuramashiSakuraiTadano2012, OLeary1980LAA, RashediEbadiBirkFrommer2016JCAM, Tadano2019JJIAM}), and such a stabilized implementation is currently the standard approach.

\subsection{On the residual gap}\label{sec1.4}

An important issue with block-type solvers is the amplification of the residual gap, which can lead to a loss of attainable accuracy of the approximations. 
This issue is examined in detail below.

For updating the approximation $X_k$ and the corresponding residual $R_k$, the block-type solvers typically use the recursion formulas
\begin{align}
	X_{k+1} = X_k + P_k \alpha_k^\square, \quad R_{k+1} = R_k - (AP_k) \alpha_k^\square, \quad k = 0,1,\dots, \label{Rec_X_R_P}
\end{align}
respectively, where $P_k \in \mathbb{R}^{n\times s}$ is a direction matrix and $\alpha_k^\square \in \mathbb{R}^{s \times s}$ is determined under a certain condition such as orthogonality or minimization. 
Given the initial residual $R_0$ ($= B - A X_0$), the equality $R_k = B - A X_k$ holds for $k = 1, 2, \dots$, in exact arithmetic, but does not necessarily hold in finite precision arithmetic due to rounding errors. 
Let us introduce the residual gap as follows:
\begin{align}
G_{R_k} := (B-AX_k) - R_k, \label{G_R_k}
\end{align}
where $B - A X_k$ and $R_k$ are the explicitly computed residual (also called true residual) and recursively updated residual, respectively. 
Then, the true residual norm can be bounded as follows: 
\begin{align}
	\| G_{R_k} \| - \| R_k \| \leq \| B-AX_k \| \leq \| G_{R_k} \| + \| R_k \|. \label{TRR}
\end{align}	
Throughout, $\| \cdot \|$ denotes the Frobenius norm, that is, $\| X \| := \sqrt{\langle X, X \rangle_F}$ with an inner product $\langle X, Y \rangle_F := \text{tr}(X^\top Y)$ for $X, Y \in \mathbb{R}^{n\times s}$. 
The bounds \eqref{TRR} imply that the attainable accuracy of $X_k$ measured by the true residual norm depends on $\|G_{R_k}\|$ when $\|R_k\|$ becomes sufficiently small. 
In this study, we propose a refined residual smoothing scheme to suppress an increase in the residual gap and to improve attainable accuracy. 

The residual gap for \eqref{Rec_X_R_P} was evaluated in our study \cite{AiharaImakuraMorikuni2022SIMAX} assuming that the columns of the direction matrix $P_k$ are exactly orthonormalized. 
The aforementioned orthonormalization strategy for block-type solvers is related to the residual gap (and our residual smoothing scheme described later). 
In the following, the column-orthonormal matrix is expressed as $Q_k$ rather than $P_k$, $\mathbb{F} \subset \mathbb{R}$ is a set of floating point numbers, fl$(\cdot)$ denotes the result of floating point operations, $\textbf{u}$ denotes the unit roundoff, and $m$ is the maximum number of nonzero entries per row of $A$. 

\begin{theorem}[{\cite[Theorem~2.2]{AiharaImakuraMorikuni2022SIMAX}}]\label{Th1}
In finite precision arithmetic, let $X_k \in \mathbb{F}^{n \times s}$ and $R_k \in \mathbb{F}^{n \times s}$ be the $k$th approximation and residual, respectively, generated by 
\begin{align}
	X_{k+1} = X_k + Q_k \alpha_k^\square, \quad R_{k+1} = R_k - (AQ_k) \alpha_k^\square, \quad k=0,1,\dots, \label{Rec_X_R_Q}
\end{align}
where $X_0 := O$. 
If the columns of $Q_j \in \mathbb{R}^{n \times s}$ are (exactly) orthonormal for $j < k$, the norm of the residual gap $G_{R_k} = (B-A X_k) - R_k$ is bounded as 
\begin{align*}
	\| G_{R_k} \| \lesssim k (3 + 4s\sqrt{s} + 2m\sqrt{s}) \mathbf{u} \| A \| \max_{0<j\leq k} \| X_j \| + 3 (k+1) \mathbf{u} \max_{0\leq j \leq k} \| R_j \|.
\end{align*}
Here, the inequality $\lesssim$ is due to the omission of $\mathcal{O}(\mathbf{u}^2)$ and regarding $\mathbf{u} \| \mathrm{fl} (\cdot) \|$ as $\mathbf{u} \| \cdot \|$. 
\end{theorem}

This theorem implies that large approximation and residual norms contribute to an increase in the residual gap, as is often observed in practice. 
However, the exact orthonormality assumption for $Q_j$ is scarcely satisfied in finite precision arithmetic. 
As noted in \cite{AiharaImakuraMorikuni2022SIMAX}, the bounds on the residual gap remain unclear when performing an inexact orthonormalization to direction matrices. 
Although it was also remarked in several papers (e.g., see \cite{KuramotoTadano2020TJSIAM, NakamuraIshikawaKuramashiSakuraiTadano2012, Tadano2019JJIAM}) that the orthonormalization for iteration matrices is effective for reducing the residual gap, only qualitative observations were made and this improvement has not been examined quantitatively via rounding error analysis. 
Therefore, this study also aims to provide a quantitative evaluation of the residual gap when inexactness is allowed in the orthonormalization; the evaluation will be described in \Cref{add_th}.

\subsection{Advances in residual smoothing}

For the Lanczos-type solvers for a single linear system, residual smoothing introduced by Sch\"{o}nauer~\cite{Schonauer1987} is a well-known technique to obtain a non-increasing sequence of residual norms and has been investigated deeply by Weiss~\cite{Weiss1994NLAA} and Walker~\cite{Walker1995APNUM}. 
This technique can also be used to connect different solvers~\cite{Weiss1996}. 
For instance, Zhou and Walker~\cite{ZhouWalker1994SISC} clarified the connection between BiCG and QMR via residual smoothing and developed the so-called QMR smoothing scheme. 
In~\cite{ZhouWalker1994SISC}, an alternative smoothing implementation was also suggested to enhance numerical stability. 
Further, the standard residual smoothing was extended for multiple right-hand sides. 

Here, we introduce a block version of the simple residual smoothing (Bl-SRS) presented by Jbilou~\cite{Jbilou1999JCAM} (we refer to Zhang and Dai~\cite{ZhangDai2008NMJCU} for the global version). 
Let $\{X_k\}$ and $\{R_k\}$ denote the primary sequences of approximations and residuals, respectively. 
With $Y_0 := X_0$ and $S_0 := R_0$, new sequences of approximations $Y_k$ and the corresponding smoothed residuals $S_k$ ($= B - AY_k$) are generated by the recursion formulas 
\begin{align}
Y_k = Y_{k-1} + (X_k - Y_{k-1})\eta_k^\square, \quad S_k = S_{k-1} + (R_k - S_{k-1}) \eta_k^\square,\quad k=0,1,\dots, \label{Bl-SRS} 
\end{align}
where $\eta_k^\square \in \mathbb{R}^{s\times s}$ is a parameter matrix of residual smoothing. 
Here, the parameter matrix is typically determined by locally minimizing the residual norm 
\begin{align}
	\eta_k^\square = \arg \min \|S_k\| = \arg \min \|S_{k-1} + E_k \eta_k^\square \| = -(E_k^\top E_k)^{-1}(E_k^\top S_{k-1}), \label{Bl-eta}
\end{align}
where $E_k := R_k - S_{k-1}$ is assumed to have full column rank. 
Subsequently, the inequality $\|S_k\| \leq \min(\|R_k\|, \|S_{k-1}\|)$ holds, and we obtain a smooth convergence behavior of the residual norms. 

There is extensive literature on the relations between residual smoothing and residual gap. 
For a single right-hand side case, Gutknecht and Rozlo\v{z}n\'{i}k~\cite{GutknechtRozloznik2001BIT} clarified that the conventional smoothing schemes (including the Zhou--Walker implementation) do not help in improving the attainable accuracy. 
Specifically, rounding errors accumulated in the primary sequences propagate to the smoothed sequences, and the smoothed true residual norms stagnate at the same order of magnitude as the primary ones. 
To remedy this phenomenon, Komeyama et al.~\cite{AiharaKomeyamaIshiwata2019BIT, KomeyamaAiharaIshiwata2018TJSIAM} modified the Zhou--Walker implementation to ensure that the primary and smoothed sequences influence one another. 
This modification is referred to as cross-interactive residual smoothing (CIRS) and can reduce the residual gap and result in a higher attainable accuracy. 
As SRS has been extended to global and block versions~\cite{ZhangDai2008NMJCU,Jbilou1999JCAM}, CIRS is also extended. 
Our previous study~\cite{AiharaImakuraMorikuni2022SIMAX} presented a global version of CIRS (Gl-CIRS) for the global- and block-type solvers, and therefore, we propose a block version of CIRS (Bl-CIRS) in this study. 
In contrast to Gl-CIRS, Bl-CIRS requires attention to the conditioning of iteration matrices to be inverted for numerical stability. 
A common remedy is to equivalently reformulate the algorithm to orthonormalize the columns of iteration matrices. 
Bl-CIRS combined with such orthonormalization uses recursion formulas based on the forms \eqref{Rec_X_R_Q} to update the approximations and the corresponding smoothed residuals as follows: 
\begin{align}
Y_{k+1} = Y_k + \tilde Q_{k+1} \tilde\eta_{k+1}^\square, \quad S_{k+1} = S_k - (A\tilde Q_{k+1}) \tilde \eta_{k+1}^\square, \quad k = 0,1,\dots, \label{Y_S_CIRS}
\end{align}
where $\tilde Q_{k+1} \in \mathbb{R}^{n\times s}$ is a column-orthonormal matrix and $\tilde \eta_{k+1}^\square \in \mathbb{R}^{s\times s}$ is a redefined parameter matrix; for details, refer to \cref{sec2.2}. 

\Cref{Deff_res} summarizes the recursion formulas for the aforementioned residual smoothing schemes. 
This table shows that this study represents a culmination of the CIRS schemes. 
Since Bl-CIRS can be directly derived from the above relationships, our focus is primarily on its theoretical foundation and numerical support. 
Additionally, we evaluate its residual gap when implemented in conjunction with the orthonormalization strategy.

\begin{table}[t]
\centering
\caption{Difference in the recursion formulas for updating smoothed residuals.}\label{Deff_res}
\begin{tabular}{c|c|c} \toprule
\bf Type       						& \bf SRS scheme & \bf CIRS scheme \\ \midrule
\multirow{2}{*}{Standard} 		& cf.~\cite{Schonauer1987,Weiss1996} & cf.~\cite{AiharaKomeyamaIshiwata2019BIT,KomeyamaAiharaIshiwata2018TJSIAM} \\ 
\multirow{2}{*}{($A\bm{x}=\bm{b}$)} & $\bm{s}_k = \bm{s}_{k-1} + \eta_k(\bm{r}_k - \bm{s}_{k-1})$, & $\bm{s}_k = \bm{s}_{k-1} - \eta_k A\bm{v}_k$, \\
& $\bm{r}_k, \bm{s}_{k-1} \in \mathbb{R}^n,\quad \eta_k \in \mathbb{R}$ & $\bm{v}_k, \bm{s}_{k-1} \in \mathbb{R}^n,\quad \eta_k \in \mathbb{R}$\\ \midrule
\multirow{2}{*}{Global} & cf.~\cite{ZhangDai2008NMJCU} & cf.~\cite{AiharaImakuraMorikuni2022SIMAX} \\ 
\multirow{2}{*}{($AX=B$)} & $S_k = S_{k-1} + \eta_k (R_k - S_{k-1})$, & $S_k = S_{k-1} - \eta_k AV_k$, \\
& $R_k, S_{k-1} \in \mathbb{R}^{n\times s},\quad \eta_k \in \mathbb{R}$ & $V_k, S_{k-1} \in \mathbb{R}^{n\times s},\quad \eta_k \in \mathbb{R}$\\ \midrule
\multirow{2}{*}{Block} & cf.~\cite{Jbilou1999JCAM} & [Present study] \\ 
\multirow{2}{*}{($AX=B$)} & $S_k = S_{k-1} + (R_k - S_{k-1})\eta_k^\square$, & $S_k = S_{k-1} - (A\tilde Q_k)\tilde\eta_k^\square$, \\
& $R_k, S_{k-1} \in \mathbb{R}^{n\times s},\quad \eta_k^\square \in \mathbb{R}^{s\times s}$ & $\tilde Q_k, S_{k-1} \in \mathbb{R}^{n\times s},\quad \tilde\eta_k^\square \in \mathbb{R}^{s\times s}$\\ \bottomrule
\end{tabular}
\end{table}

\subsection{Organization}

The remainder of this paper is organized as follows. 
\Cref{sec2} describes an underlying Bl-CIRS scheme and proposes its refined implementation combined with an orthonormalization strategy. 
\Cref{sec3} conducts numerical experiments to demonstrate the effectiveness of the proposed approach. 
In \cref{sec4}, we present a rounding error analysis for the residual gap when using Bl-CIRS with the orthonormalization strategy. 
Finally, concluding remarks are given in \cref{sec5}.

\section{Block cross-interactive residual smoothing}\label{sec2}

In this section, we present Bl-CIRS for reducing the residual gap of the Lanczos-type solvers for \eqref{eq:AX=B}. 
All the discussions in this section are assumed to be in exact arithmetic unless otherwise noted.

\subsection{Underlying Bl-CIRS scheme}

We briefly describe an underlying Bl-CIRS scheme that is a natural extension of the standard CIRS \cite[Algorithm~1]{AiharaKomeyamaIshiwata2019BIT} and Gl-CIRS \cite[Algorithm~3.1]{AiharaImakuraMorikuni2022SIMAX} to the block version.

Let $\tilde P_k \in \mathbb{R}^{n\times s}$ be the difference $X_{k+1}-X_k$ between adjacent approximations in the primary method. 
Note that $\tilde P_k$ also represents a direction matrix in the recursion formula $X_{k+1} = X_k + \tilde{P}_k$. 
After generating such a direction matrix $\tilde P_k$, we update an auxiliary matrix $V_k$ recursively as follows: 
\begin{align*}
V_{k+1} = V_k(I_s - \eta_k^\square) + \tilde P_k, 
\end{align*}
where $\eta_0^\square := O \in \mathbb{R}^{s\times s}$ and $V_0 := O \in \mathbb{R}^{n\times s}$. 
Subsequently, explicitly multiplying $V_{k+1}$ by $A$ gives another auxiliary matrix $U_{k+1} := AV_{k+1}$.  
We next update an approximation and the corresponding smoothed residual by the standard recursion formulas 
\begin{align*}
Y_{k+1} = Y_k + V_{k+1}\eta_{k+1}^\square,\quad S_{k+1} = S_k - U_{k+1}\eta_{k+1}^\square, 
\end{align*}
respectively, where $Y_0 := X_0$ and $S_0 := R_0$. 
Similarly to \eqref{Bl-eta}, a local minimization of $\|S_{k+1}\|$ gives $\eta_{k+1}^\square = (U_{k+1}^\top U_{k+1})^{-1}(U_{k+1}^\top S_k)$, where $U_{k+1}$ is assumed to have full column rank. 
We finally compute the primary approximation and residual by 
\begin{align*}
X_{k+1} = Y_{k+1} + V_{k+1}(I_s - \eta_{k+1}^\square),\quad R_{k+1} = S_{k+1} - U_{k+1}(I_s - \eta_{k+1}^\square), 
\end{align*}
and return them to the primary method. 

\Cref{alg1} displays Bl-CIRS above. 
By induction, it is shown that $Y_k$ and $S_k$ generated by \Cref{alg1} are the same as those generated by \eqref{Bl-SRS} with \eqref{Bl-eta}. 
Thus, $\|S_k\| \leq \min(\|R_k\|, \|S_{k-1}\|)$ holds also for Bl-CIRS, and a monotonically decreasing sequence of $\|S_k\|$ can be obtained. 
The essential difference of Bl-CIRS (\Cref{alg1}) from Gl-CIRS \cite[Algorithm~3.1]{AiharaImakuraMorikuni2022SIMAX} (or \cite[Algorithm~1]{AiharaKomeyamaIshiwata2019BIT}) is that the smoothing parameter of Bl-CIRS is an $s$-by-$s$ matrix $\eta_k^\square$ rather than a scalar. 
For details of CIRS, for instance, its derivation and properties, we refer the reader to \cite{AiharaImakuraMorikuni2022SIMAX, AiharaKomeyamaIshiwata2019BIT}.

\begin{algorithm}[t]
\caption{Underlying Bl-CIRS scheme.}\label{alg1}
\begin{algorithmic}[1]
\REQUIRE An initial guess $X_0$ and the initial residual $R_0:=B-AX_0$. 
\STATE Set $Y_0 := X_0$, $S_0 := R_0$, $V_0 := O$, and $\eta_0^\square := O$.
\FOR{$k=0,1,2,\dots$ until convergence}
	\STATE Compute $\tilde P_k$ (corresponding to $X_{k+1}-X_k$) using the primary method.
	\STATE $V_{k+1} = V_k(I_s - \eta_k^\square) + \tilde P_k$
	\STATE Compute $U_{k+1} = AV_{k+1}$ with an explicit multiplication by $A$.
	\STATE $\eta_{k+1}^\square = (U_{k+1}^\top U_{k+1})^{-1}(U_{k+1}^\top S_k)$
	\STATE $Y_{k+1} = Y_k + V_{k+1}\eta_{k+1}^\square,\quad S_{k+1} = S_k - U_{k+1}\eta_{k+1}^\square$
	\RETURN $X_{k+1} = Y_{k+1} + V_{k+1}(I_s - \eta_{k+1}^\square),\quad R_{k+1} = S_{k+1} - U_{k+1}(I_s - \eta_{k+1}^\square)$
\ENDFOR
\end{algorithmic}
\end{algorithm}

\subsection{Orthonormalization strategy}\label{sec2.2}

For the numerical stability of Bl-CIRS, we suggest orthonormalizing the columns of the auxiliary matrix $V_k \in \mathbb{R}^{n\times s}$ for each iteration. 
Hereafter, assume that $V_k$ has full column rank. 

Let $V_k = \tilde Q_k \tilde \xi_k$ be the (thin) QR decomposition, where $\tilde Q_k \in \mathbb{R}^{n\times s}$ is a column-orthonormal matrix and $\tilde \xi_k \in \mathbb{R}^{s\times s}$ is the associated upper triangular matrix. 
Then, introducing auxiliary matrices $\tilde U_k := A\tilde Q_k$ and $\tilde \eta_k^\square := \tilde \xi_k \eta_k^\square$, the approximations and residuals in \Cref{alg1} are rewritten as 
\begin{align*}
&Y_k = Y_{k-1} + V_k \eta_k^\square = Y_{k-1} + \tilde Q_k\tilde \eta_k^\square, \\
&S_k = S_{k-1} - U_k \eta_k^\square = S_{k-1} - \tilde U_k\tilde \eta_k^\square, \\
&X_k = Y_k + V_k (I_s - \eta_k^\square) = Y_k + \tilde Q_k(\tilde \xi_k - \tilde \eta_k^\square), \\
&R_k = S_k - U_k (I_s - \eta_k^\square) = S_k - \tilde U_k(\tilde \xi_k - \tilde \eta_k^\square),
\end{align*}
where the quantity $\tilde U_k := A\tilde Q_k$ ($=(S_{k-1} - R_k)\tilde \xi_k^{-1}$) is given by explicitly multiplying $\tilde Q_k$ by $A$ instead of $U_k := AV_k$ ($=S_{k-1} - R_k$). 
Similarly, the next auxiliary matrix can be expressed by 
\begin{align*}
(\tilde Q_{k+1}\tilde \xi_{k+1} =)\ V_{k+1} = V_k(I_s -\eta_k^\square) + \tilde P_k= \tilde Q_k(\tilde \xi_k - \tilde \eta_k^\square) + \tilde P_k. 
\end{align*}
The redefined smoothing parameter $\tilde \eta_k^\square$ is computed by 
\begin{align*} 
\tilde \eta_k^\square 
&:= \tilde \xi_k\eta_k^\square = \tilde \xi_k(U_k^\top U_k)^{-1}(U_k^\top S_{k-1}) \\
&= [(U_k\tilde \xi_k^{-1})^\top U_k\tilde \xi_k^{-1}]^{-1}[(U_k\tilde \xi_k^{-1})^\top S_{k-1}] = (\tilde U_k^\top \tilde U_k)^{-1}(\tilde U_k^\top S_{k-1}), 
\end{align*}
or equivalently, the final term comes from a local minimization of the smoothed residual norm, i.e., $\tilde \eta_k^\square = \arg \min \|S_{k-1} - \tilde U_k \tilde \eta_k^\square \|$. 
Here, the condition number of $\tilde{U}_k$ is bounded above by that of $A$. 

\Cref{alg2} displays Bl-CIRS combined with the orthonormalization strategy.~Here,\break $[\cdot, \cdot] = \textbf{qr}(\cdot )$ denotes the QR decomposition of a matrix in parentheses; the first and second variables on the left-hand side are the Q- and R-factors, respectively. 
The approximations and corresponding smoothed residuals are updated by the forms~\eqref{Y_S_CIRS}, and these forms are crucial in reducing the residual gap; refer to the next section and \cref{subsec3.2}.

\begin{algorithm}[t]
\caption{Refined Bl-CIRS scheme with orthonormalization.}\label{alg2}
\begin{algorithmic}[1]
\REQUIRE An initial guess $X_0$ and the initial residual $R_0:=B-AX_0$. 
\STATE Set $Y_0 := X_0$, $S_0 := R_0$, $\tilde Q_0 := O$, $\tilde \xi_0 := O$, and $\tilde \eta_0^\square := O$.
\FOR{$k=0,1,2,\dots$ until convergence}
	\STATE Compute $\tilde P_k$ (corresponding to $X_{k+1}-X_k$) using the primary method.
	\STATE $V_{k+1} = \tilde Q_k(\tilde \xi_k - \tilde\eta_k^\square) + \tilde P_k$
	\STATE Compute QR decomposition: $[\tilde Q_{k+1}, \tilde \xi_{k+1}] = \textbf{qr}(V_{k+1})$. 
	\STATE Compute $\tilde U_{k+1} = A\tilde Q_{k+1}$ with an explicit multiplication by $A$.
	\STATE $\tilde\eta_{k+1}^\square = (\tilde U_{k+1}^\top \tilde U_{k+1})^{-1}(\tilde U_{k+1}^\top S_k)$
	\STATE $Y_{k+1} = Y_k + \tilde Q_{k+1}\tilde\eta_{k+1}^\square,\quad S_{k+1} = S_k - \tilde U_{k+1}\tilde\eta_{k+1}^\square$
	\RETURN $X_{k+1} = Y_{k+1} + \tilde Q_{k+1}(\tilde \xi_{k+1} - \tilde\eta_{k+1}^\square),\quad R_{k+1} = S_{k+1} - \tilde U_{k+1}(\tilde \xi_{k+1} - \tilde\eta_{k+1}^\square)$
\ENDFOR
\end{algorithmic}
\end{algorithm}

\subsection{Effect of approximation norms on the residual gap}\label{subsec2.3}

The residual gap $G_{R_k}$ \eqref{G_R_k} in \Cref{Th1} reduces to 
\begin{align}
	\| G_{R_k} \| < c_x \| A \| \max_{0<j\leq k} \| X_j \| + c_r \max_{0\leq j \leq k} \| R_j \|, \label{G_X_R}
\end{align}
where the coefficients $c_x$ and $c_r$ are small positive scalars of order $\mathbf{u}$ (see also the analysis in \cref{sec4} and \cite{Greenbaum1997SIMAX, GutknechtRozloznik2001BIT}). 
Because Bl-CIRS with the orthonormalization strategy uses the recursion formulas \eqref{Y_S_CIRS} that are essentially the same as \eqref{Rec_X_R_Q}, the bound \eqref{G_X_R} also holds for the smoothed case; that is, $X_j$, $R_j$, and $G_{R_k}$ are replaced by $Y_j$, $S_j$, and $G_{S_k} := (B-AY_k) - S_k$, respectively. 
Note that $\|S_j\|$ decreases monotonically, whereas $\|R_j\|$ can oscillate significantly. 
The residual gap can be suppressed by suppressing $\max_j \|R_j\|$ in the second term in \eqref{G_X_R}. 
However, the bound also depends on $\max_j \|X_j\|$. 
Therefore, we discuss the effect of the residual smoothing on the convergence behavior of the approximation norms. 

From the recursion formula $Y_{k+1} = Y_k + \tilde Q_{k+1} \tilde\eta_{k+1}^\square$, the norm of the updated approximation is bounded as follows: 
\begin{align}
\|Y_k\| - \|\tilde\eta_{k+1}^\square \| \leq \|Y_{k+1}\| \leq \|Y_k\| +  \|\tilde\eta_{k+1}^\square \|. \label{Y_eta}
\end{align}
As $\|S_{k+1}\| \leq \|S_k\|$ holds, we have
\begin{align*}
\|\tilde\eta_{k+1}^\square \| = \|A^{-1}(S_k - S_{k+1})\|  \leq 2 \|A^{-1}\| \|S_k\|, 
\end{align*}
and the above bounds lead to the following inequalities: 
\begin{align*}
\frac{\|Y_k\|}{\|B\|} - 2\tau_k \leq \frac{\|Y_{k+1}\|}{\|B\|} \leq \frac{\|Y_k\|}{\|B\|} + 2\tau_k,\quad \tau_k := \|A^{-1}\| \frac{\|S_k\|}{\|B\|}.
\end{align*}
Here, $\|S_k\|/\|B\|$ ($\leq 1$ for $Y_0 := O$) decreases monotonically with increasing $k$. 
Therefore, if $\|A^{-1}\|$ is modest, then the difference between $\|Y_k\|$ and $\|Y_{k+1}\|$ is also modest and becomes smaller as the iteration proceeds. 
Even at the early stage of iterations, because $\|\eta_k^\square\|$ is small for a large $\|R_k\|$ from \eqref{Bl-SRS}, there may be no large oscillations in $\|Y_k\|$ from \eqref{Y_eta}. 
Thus, the sequence of the approximation norms starting with $\|Y_0\| = 0$ can have a smooth (not necessarily monotonic) convergence behavior in the generic case. 

However, from the equality $Y_k = X^* - A^{-1}S_k$ for $X^* := A^{-1}B$, an upper bound of the approximation norms is given as follows: 
\begin{align*}
\frac{\|Y_k\|}{\|B\|}  \leq \frac{\|X^*\|}{\|B\|} + \tau_k. 
\end{align*}
Therefore, we expect that $\|Y_k\|$ will remain small relative to $\|X^*\|$. 
In the methods without residual smoothing, $S_k$ is replaced by $R_k$ in the definition of $\tau_k$, and the upper bound of $\|X_k\|$ can be magnified by the oscillations in $\|R_k\|$. 

In actual computation, as will be shown in \cref{subsec3.2.2}, $\|Y_k\|$ increases smoothly and approaches $\|X^*\|$, whereas $\|X_k\|$ can oscillate significantly during the iterations. 
This is a strength of the residual smoothing from the perspective of the approximation norms and contributes to preventing a large residual gap.

\subsection{Specific Lanczos-type solver with Bl-CIRS}

We present a specific Lanczos-type solver combined with the proposed Bl-CIRS with the orthonormalization strategy. 
We apply \Cref{alg2} to the Bl-BiCGSTABpQ method, which is a stabilized variant of the well-established method Bl-BiCGSTAB and orthonormalizes the columns of iteration 
matrices \cite{GuennouniJbilouSadok2003ETNA, NakamuraIshikawaKuramashiSakuraiTadano2012}. 
Our derivation below follows those of the smoothed Gl- and Bl-BiCGSTABpQ combined with Gl-CIRS \cite[sections~5.2 and 5.3]{AiharaImakuraMorikuni2022SIMAX}; the resulting algorithm is an extension of \cite[Algorithm~5.3]{AiharaImakuraMorikuni2022SIMAX} for using Bl-CIRS. 

As noted in \cite{AiharaImakuraMorikuni2022SIMAX}, Bl-BiCGSTABpQ updates the approximation $X_k$ and the corresponding residual $R_k$ by two steps at each iteration as follows: 
\begin{align}
\begin{split}
&\text{(BiCG part)}\quad X_k' := X_k + Q_k\alpha_k^\square,\quad R_k' := R_k - (AQ_k)\alpha_k^\square,\\
&\text{(polynomial part)}\quad X_{k+1} = X_k' + \omega_k R_k',\quad R_{k+1} = R_k' - \omega_k (AR_k'), 
\end{split}\label{G_B_X_R}
\end{align}
where $\omega_k \in \mathbb{R}$ is normally determined by minimizing $\|R_{k+1}\|$ and $Q_k \in \mathbb{R}^{n\times s}$ (a column-\hspace{0pt}orthonormalized matrix) is the Q-factor of the QR decomposition of the direction matrix; see below for $\alpha_k^\square \in \mathbb{R}^{s\times s}$. 
It is simple to apply Bl-CIRS to both the residuals in the BiCG and polynomial parts, but this approach requires two additional multiplications by $A$ per iteration. 
Therefore, following the concept in \cite{AiharaImakuraMorikuni2022SIMAX,AiharaKomeyamaIshiwata2019BIT}, we apply Bl-CIRS only to the BiCG part residuals to reduce the computational costs. 
Specifically, we reformulate \eqref{G_B_X_R} as follows: 
\begin{align*}
X_k' = X_{k-1}' + \tilde P_{k-1}, \quad R_k' = R_{k-1}' - A\tilde P_{k-1},\quad \tilde P_{k-1}: = \omega_{k-1} R_{k-1}' + Q_k\alpha_k^\square,
\end{align*}
and apply \Cref{alg2} to the primary sequences $\{X_k'\}$ and $\{R_k'\}$, where $X_{-1}' := X_0$, $R_{-1}' := O$, and $\omega_{-1} := 0$. 
The setting of $\tilde P_k$ is variable depending on the primary sequence to which Bl-CIRS is applied. 
Moreover, to reduce the number of multiplications by $A$, we obtain $\alpha_k^\square$ as a solution of an $s$-dimensional linear system $({Z_0^{\bullet}}^\top Q_k) \alpha_k^\square ={R_0^{\bullet}}^\top R_k$, where $Z_0^{\bullet} := A^\top R_0^{\bullet}$ is computed once and stored in advance. 
Although $AQ_k$ is needed to compute the next direction matrix, it can be provided by solving the system $(AQ_k) \alpha_k^\square = R_k - R_k'$, after $R_k'$ is returned from Bl-CIRS. 
The resulting algorithm requires no additional multiplications by $A$. 

\Cref{alg3} displays the smoothed Bl-BiCGSTABpQ using Bl-CIRS with the orthonormalization strategy. 
In line~6, $\textbf{qf}(\cdot )$ denotes the Q-factor of the QR decomposition of a matrix. 
This algorithm requires two orthonormalizations per iteration; one is for the direction matrix $P_k$ in Bl-BiCGSTABpQ and another is for the auxiliary matrix $V_k$ in Bl-CIRS. 
Lines~8--12 correspond to Bl-CIRS, where $\tilde \zeta_k := \tilde \xi_k - \tilde \eta_k^\square$ is introduced for efficiency. 
The indices of variables (e.g., the number of iterations $k$) are omitted in the actual implementation: however, several variables that can be overwritten by others remain for clarity and readability. 
For instance, the storage for $n$-by-$s$ matrices can be reduced by overwriting the orthonormalized matrices $Q$ and $\tilde Q$ with $P$ and $V$, respectively. 
Moreover, we do not need to compute $X$ and $X'$ if satisfied with $Y$ to capture approximations for \eqref{eq:AX=B}.

\begin{algorithm}[t]
\caption{Smoothed Bl-BiCGSTABpQ using Bl-CIRS with orthonormalization.}\label{alg3}
\begin{algorithmic}[1]
\STATE Select an initial guess $X$ and compute the initial residual $R = B-AX$.
\STATE Set $Y := X$, $S := R$, $\tilde Q := O$, and $\tilde \zeta := O$.
\STATE Select an initial shadow residual $R_0^{\bullet}$ and compute $Z_0^{\bullet} = A^\top R_0^{\bullet}$.
\STATE Set $P := R$, $R' := O$, and $\omega := 0$.
\WHILE{$\|S\| > tol$}
	\STATE $Q = \textbf{qf}(P),\quad \sigma = {Z_0^{\bullet}}^\top Q$
	\STATE Solve $\sigma \alpha = {R_0^{\bullet}}^\top R$ for $\alpha$.
	\STATE $\tilde P = \omega R' + Q\alpha,\quad V = \tilde Q \tilde \zeta + \tilde P$
	\STATE $[\tilde Q, \tilde \xi] = \textbf{qr}(V),\quad \tilde U = A \tilde Q$
	\STATE Solve $(\tilde U^\top \tilde U) \tilde \eta = \tilde U^\top \tilde S$ for $\tilde \eta$. 
	\STATE $Y = Y + \tilde Q \tilde \eta ,\quad S = S - \tilde U \tilde \eta,\quad \tilde \zeta = \tilde \xi - \tilde \eta$
	\STATE $X' = Y + \tilde Q \tilde \zeta,\quad R' = S - \tilde U \tilde \zeta$
	\STATE Solve $V' \alpha = R - R'$ for $V'$.
	\STATE $T = A R',\quad \omega = \langle R', T\rangle_F / \langle T, T \rangle_F$
	\STATE $X = X' + \omega R',\quad R = R' - \omega T$
	\STATE Solve $\sigma \beta = {R_0^{\bullet}}^\top T$ for $\beta$.
	\STATE $P = R - (Q - \omega V')\beta$
\ENDWHILE
\end{algorithmic}
\end{algorithm}

\section{Numerical experiments}\label{sec3}

In this section, numerical experiments demonstrate that \Cref{alg3} can effectively reduce the residual gap and improve the attainable accuracy of the approximations. 
We first describe common settings of the experiments, and then show the numerical results, dividing them into several viewpoints. 

\subsection{Computational conditions}

Numerical calculations were conducted by double-precision floating-point arithmetic on a PC (Intel Core i7-1185G7 CPU with 32 GB of RAM) equipped with MATLAB R2021a. 

The iterations were started with $X_0 := O$ and were stopped when the relative residual norms (i.e., $\|S_k\|/\|B\|$ and $\|R_k\|/\|B\|$ for the methods with and without residual smoothing, respectively) were less than $10^{-15}$. 
The maximum number of iterations was set to $n$. 
Following \cite{AiharaImakuraMorikuni2022SIMAX}, the test matrices $A$ were given from the \mbox{SuiteSparse} Matrix Collection \cite{DavisHu2011ACM}. 
\Cref{table_matrix} displays the dimension ($n$), number of nonzero entries ($nnz$), maximum number of nonzero entries per row ($m$), and 2-norm condition number ($\text{cond}_2(A)$). 
The right-hand side $B$ was given by a random matrix with setting $s = 16, 32$, and the initial shadow residual $R_0^{\bullet}$ was set to $R_0$ ($= B$). 

We used Bl-BiCGSTABpQ as an underlying solver and compared its five variations listed below.  
\begin{description}
\item[Solver~1:] No residual smoothing (the simple Bl-BiCGSTABpQ). 
\item[Solver~2:] Using Bl-SRS (\eqref{Bl-SRS} with \eqref{Bl-eta}). 
\item[Solver~3:] Using Bl-CIRS without orthonormalization (\Cref{alg1}). 
\item[Solver~4:] Using Bl-CIRS with orthonormalization (\Cref{alg2}). 
\item[Solver~5:] Using Gl-CIRS (\cite[Algorithm~5.3]{AiharaImakuraMorikuni2022SIMAX}). 
\end{description}
Note that Solver~4 is our proposed method and corresponds to \Cref{alg3}. 
In the implementation on MATLAB, the slash or backslash command was used for solving the small linear systems (e.g., lines~7, 10, 13, and 16 in \Cref{alg3}), and the QR decomposition used in the solvers was performed by the $\textbf{qr}(\cdot)$ command. 
All the solvers require two multiplications by $A$ per iteration.

\begin{table}[t]
\centering
\caption{Characteristics of test matrices \cite{DavisHu2011ACM}.}\label{table_matrix}
\begin{tabular}{lrrrr} \toprule
\bf Matrix    & $n$    & $nnz$ & $m$ & $\text{cond}_2(A)$ \\ \midrule
cdde2    & 961    & 4,681  & 5    & 5.48e+01 \\
pde2961 & 2,961 & 14,585 & 5    & 6.42e+02 \\ 
bfwa782 & 782    & 7,514  & 24  & 1.74e+03 \\ \bottomrule
\end{tabular}
\end{table}

\subsection{Effectiveness of the orthonormalization strategy}\label{subsec3.2}

Here, we present the numerical results of Solvers~1--5 applied to \eqref{eq:AX=B}. 
\Cref{Results} displays the number of iterations ({\bf Iter.}), computation time ({\bf Time (s)}),  and true relative residual norm ({\bf True~res.}) at the termination of the iterations, where the true relative residual norm was evaluated by\break $\|B-AX_k\|/\|B\|$ and $\|B-AY_k\|/\|B\|$ for Solver~1 and Solvers~2--5, respectively. 
The symbol $\dagger$ denotes no convergence within $n$ iterations. 

We observe from \Cref{Results} that the solvers, except for Solver~3, converge at a similar speed for each matrix, but Solvers~4 and 5 generate more accurate approximations than the others with respect to the true relative residual norm. 
The computation time is only a guide because the provided algorithms were implemented naively on MATLAB, and their efficiencies depend on the computational equipment and implementation. 
We will not discuss this point further. 
In the following, we focus on the residual gap and discuss the feature of Bl-CIRS for each comparison item.

\begin{table}[t]
\centering
\caption{Number of iterations, computation time, and true relative residual norm of Solvers~1--5 for test matrices.}\label{Results}
\begin{tabular}{cccccccc} \toprule
 & & \multicolumn{3}{c}{$s=16$} & \multicolumn{3}{c}{$s=32$} \\ \cmidrule{3-8}
\bf Matrix & \bf Solver & \bf Iter. & \bf Time (s) & \bf True res. & \bf Iter. & \bf Time (s) & \bf True res. \\ \midrule
\multirow{5}{*}{cdde2} & \#1 & 55 & 0.026 & 1.58e$-$13 & 39 & 0.038 & 6.91e$-$11\\ 
							& \#2 & 54 & 0.028 & 2.29e$-$13 & 38 & 0.045 & 6.91e$-$11\\
							& \#3 & 55 & 0.042 & 2.83e$-$13 & 39 & 0.054 & 4.00e$-$11\\
							& \#4 & 55 & 0.046 & 7.78e$-$15 & 38 & 0.059 & 6.71e$-$15\\
							& \#5 & 56 & 0.038 & 7.41e$-$15 & 39 & 0.050 & 6.13e$-$15\\ \midrule
\multirow{5}{*}{pde2961} & \#1 & 96 & 0.116 & 9.72e$-$11 & 64 & 0.174 & 5.05e$-$13\\ 
							& \#2 & 94 & 0.124 & 1.09e$-$10 & 62 & 0.180 & 7.20e$-$11\\
							& \#3 & 100 & 0.157 & 7.77e$-$11 & 883 & 3.298 & 2.21e$-$11\\
							& \#4 & 93 & 0.156 & 6.53e$-$14 & 60 & 0.243 & 5.33e$-$14\\
							& \#5 & 90 & 0.119 & 6.19e$-$14 & 65 & 0.204 & 5.21e$-$14\\ \midrule
\multirow{5}{*}{bfwa782} & \#1 & 61 & 0.029 & 1.21e$-$12 & 40 & 0.036 & 1.64e$-$12\\ 
							& \#2 & 61 & 0.033 & 7.44e$-$11 & 37 & 0.037 &  4.26e$-$11\\
							& \#3 & $\dagger$ & - & 4.85e$-$11 & $\dagger$ & - & 2.47e$-$11\\
							& \#4 & 62 & 0.042 & 7.69e$-$14 & 36 & 0.051 & 6.41e$-$14\\
							& \#5 & 62 & 0.037 & 7.44e$-$14 & 38 & 0.046 & 6.10e$-$14\\ \bottomrule
\end{tabular}
\end{table}

\subsubsection{Comparison of the residual gap}

We first compare the convergence behaviors of the residual norms for Solvers~1--4. 
\Cref{Ex1} shows the histories of the relative norms of the recursively updated and true residuals of Solvers~1--4 for cdde2 with $s=32$. 
The plots indicate the number of iterations on the horizontal axis versus $\log_{10}$ of the relative residual norm on the vertical axis. 

From \Cref{Ex1}, we observe the following. 
As noted in \cite{AiharaImakuraMorikuni2022SIMAX}, Solver~1 (i.e., the original Bl-BiCGSTABpQ) has a large peak of the residual norms, leading to a large residual gap; that is, the true residual norms stagnate at a certain level of accuracy while the recursively updated residual norms approach zero. 
Solvers~2--4 (i.e., the methods with residual smoothing) have a smooth convergence behavior. 
However, Solvers~2 and 3 have still a large residual gap and their true residual norms stagnate at the same order of magnitude as Solver~1, whereas only Solver~4 has a smaller residual gap and generates more accurate approximations regarding the true residual norm. 
In addition, Solver~3 (using Bl-CIRS without the orthonormalization) has numerical stability problems such that the convergence often deteriorates as shown in \Cref{Results}. 
These results suggest that orthonormalization is crucial in stabilizing Bl-CIRS and in avoiding the increase of the residual gap. 
Similar observations were obtained as well for other matrices and several right-hand sides.

\begin{figure}[t]
	\begin{minipage}{0.49\textwidth}
		\centering
		\includegraphics[scale=0.2]{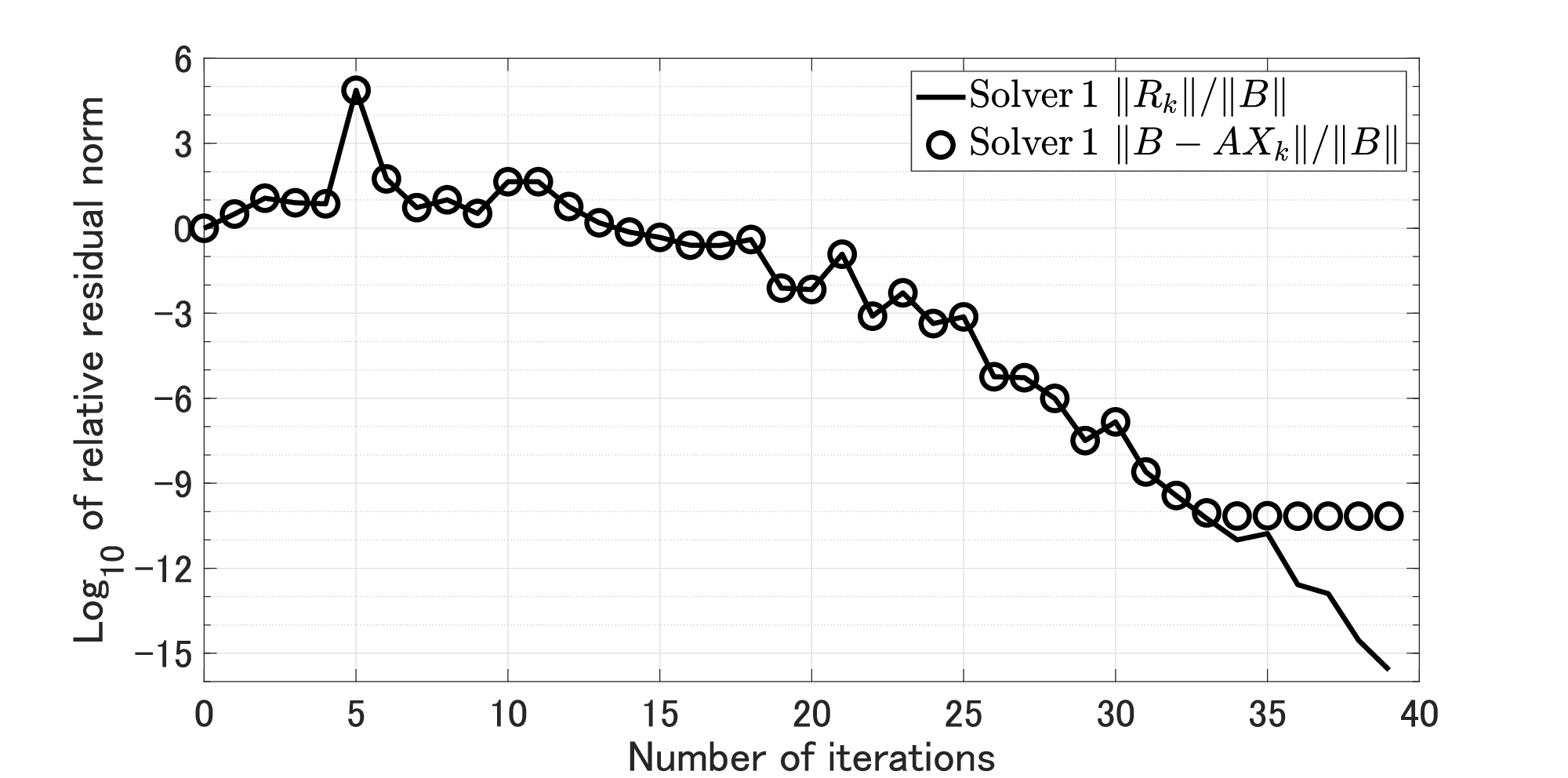}
		\subcaption{Solver~1.}\label{Ex1_1}
	\end{minipage}
	\begin{minipage}{0.49\textwidth}
		\centering
		\includegraphics[scale=0.2]{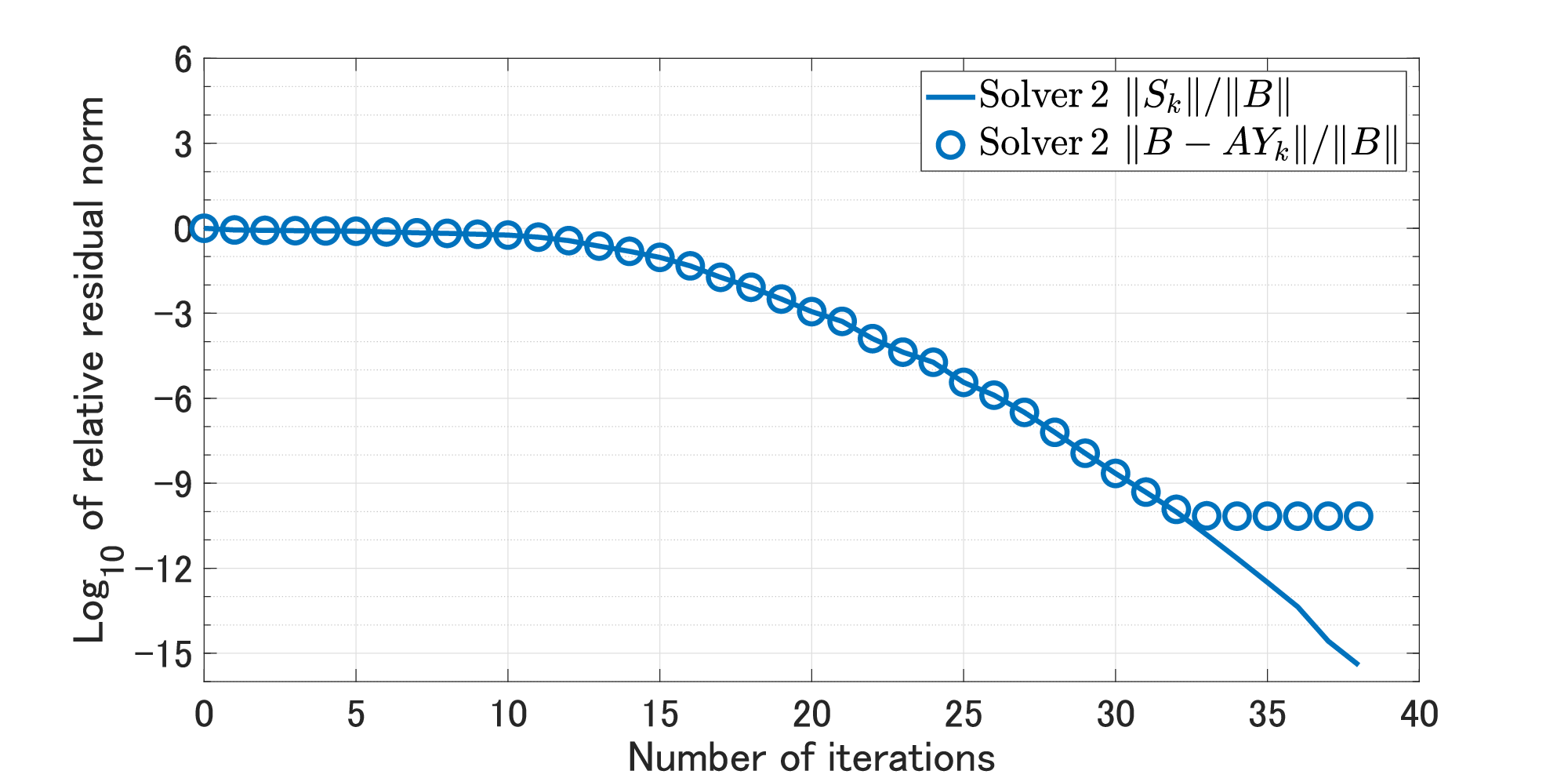}
		\subcaption{Solver~2.}\label{Ex1_2}
	\end{minipage}
	\\
	\begin{minipage}{0.49\textwidth}
		\centering
		\includegraphics[scale=0.2]{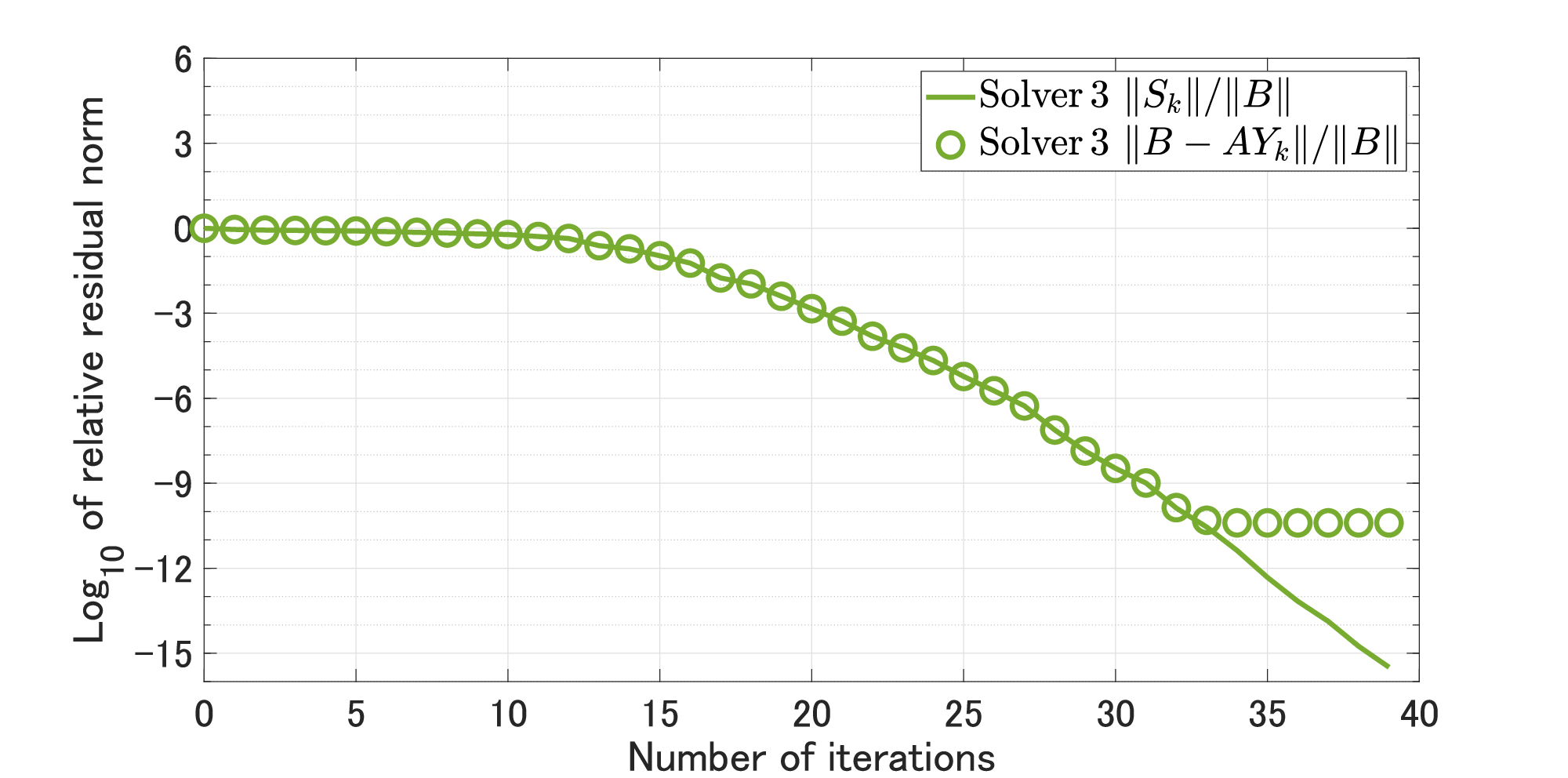}
		\subcaption{Solver~3.}\label{Ex1_3}
	\end{minipage}
	\begin{minipage}{0.49\textwidth}
		\centering
		\includegraphics[scale=0.2]{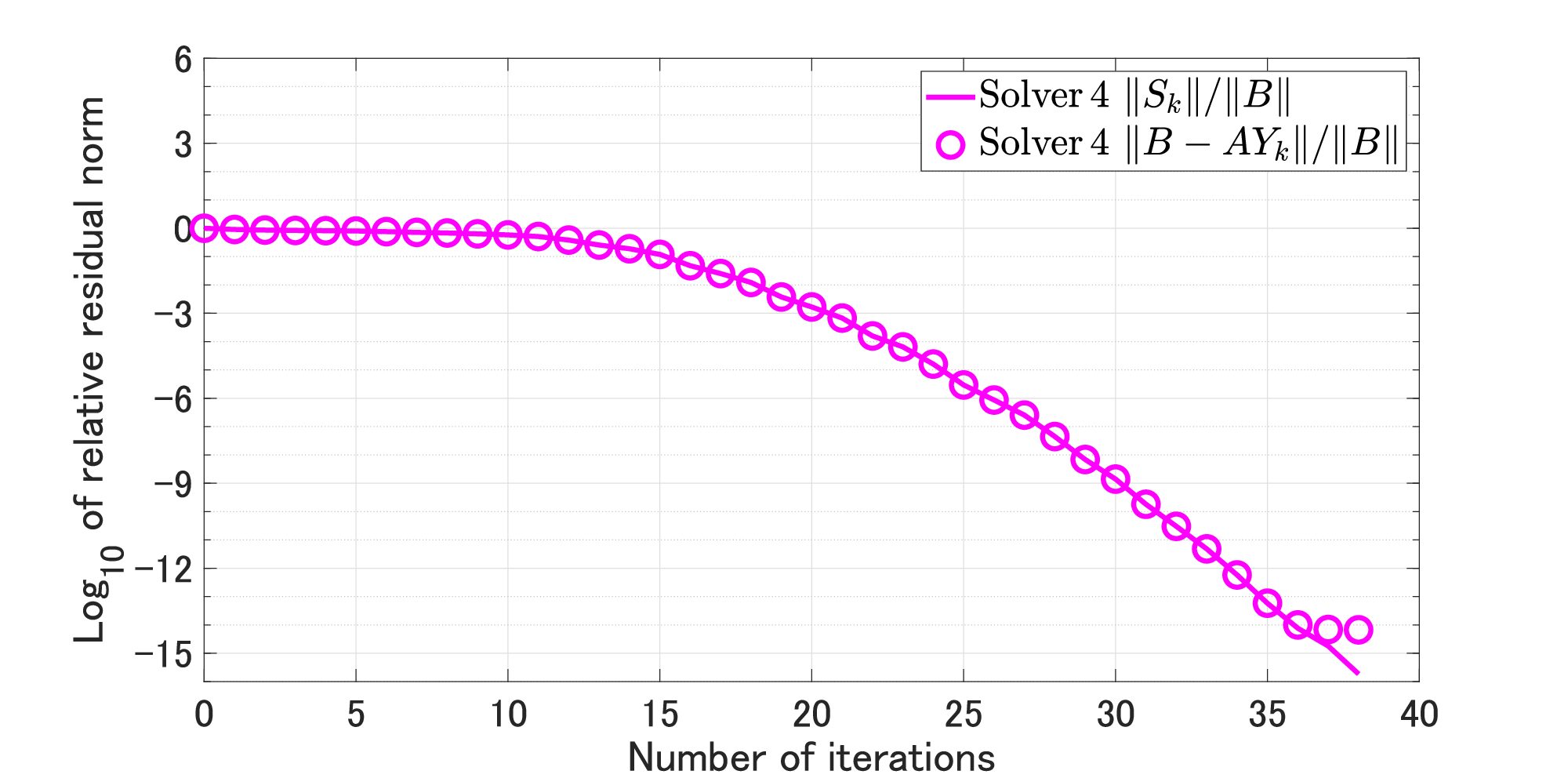}
		\subcaption{Solver~4.}\label{Ex1_4}
	\end{minipage}
	\caption{Convergence histories of Solvers~1--4 for cdde2 with $s=32$.}\label{Ex1}
\end{figure}

\subsubsection{Convergence behavior of approximation norms}\label{subsec3.2.2}

Furthermore, we examine the convergence behavior of the approximation norms with and without residual smoothing. 
In particular, we present examples in which $\max_j \|R_j\| \leq \max_j \|X_j\|$ holds and discuss the influence of the approximation norms on the residual gap. 

\Cref{Ex2_1,Ex2_2} show the convergence histories of the relative residual and approximation norms of Solvers~1 and 4 for pde2961 and bfwa782, respectively, where $s=16$. 
To see the oscillations in the residual and approximation norms precisely, the number of matrix-matrix multiplications (MMs) with $A$ is plotted on the horizontal axis, and we illustrate the norms corresponding to the Bi-CG and polynomial parts alternately for Solver~1 (cf.~\eqref{G_B_X_R}). 
The attainable accuracy, that is, the final size of the true relative residual norms at termination, is indicated with a dashed line with the same color as the corresponding residual norms. 

From \Cref{Ex2_1,Ex2_2}, we observe the following: 
Solver~4 has a smooth convergence behavior with respect to the residual norms $\|S_k\|$ but approximation norms $\|Y_k\|$, whereas Solver~1 exhibits the oscillations in $\|X_k\|$ along with $\|R_k\|$. 
As described in \Cref{Ex2_1}, the loss of attainable accuracy seems to be determined by the maximum value of the approximation norms; for instance, the final accuracy $\|B-AX_k\|/\|B\| = 9.72\times 10^{-11}$ at termination of Solver~1 corresponds to the order of $\max_j \|X_j\|/\|B\| \times 10^{-15}$ ($= 9.03\times 10^4 \times 10^{-15}$). 
This result follows the residual gap evaluation~\eqref{G_X_R}. 
In Solver~4, because it holds that $\max_j \|S_j\|/\|B\| = \|S_0\|/\|B\| = 1$ for $Y_0 := O$, the size of $\|Y_k\|$ directly affects the residual gap. 
Although $\|Y_k\|$ is not necessarily increasing monotonically as described in \Cref{Ex2_2}, it mostly behaves smoothly and the maximum value is almost the same order of magnitude as $\|X^*\|$. 
These observations follow the discussions in \cref{subsec2.3}. 
Thus, we conclude that Bl-CIRS effectively suppresses the residual gap via the residual and approximation norm. 

\begin{figure}[t]
		\centering
		\includegraphics[scale=0.3]{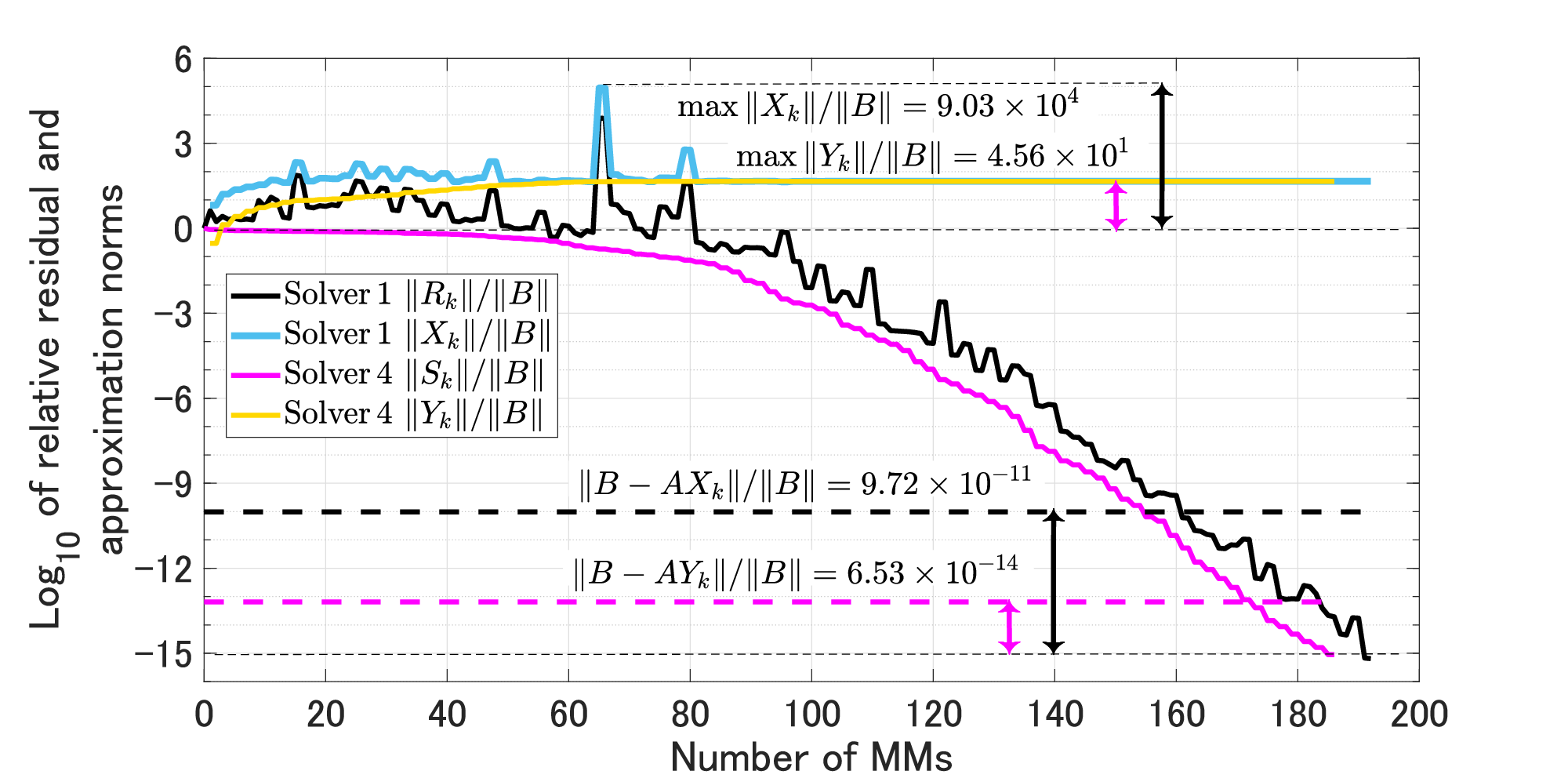}
		\caption{Convergence histories of the relative residual and approximation norms of Solvers~1 and 4 for pde2961 with $s=16$.}\label{Ex2_1}
		\centering
		\includegraphics[scale=0.3]{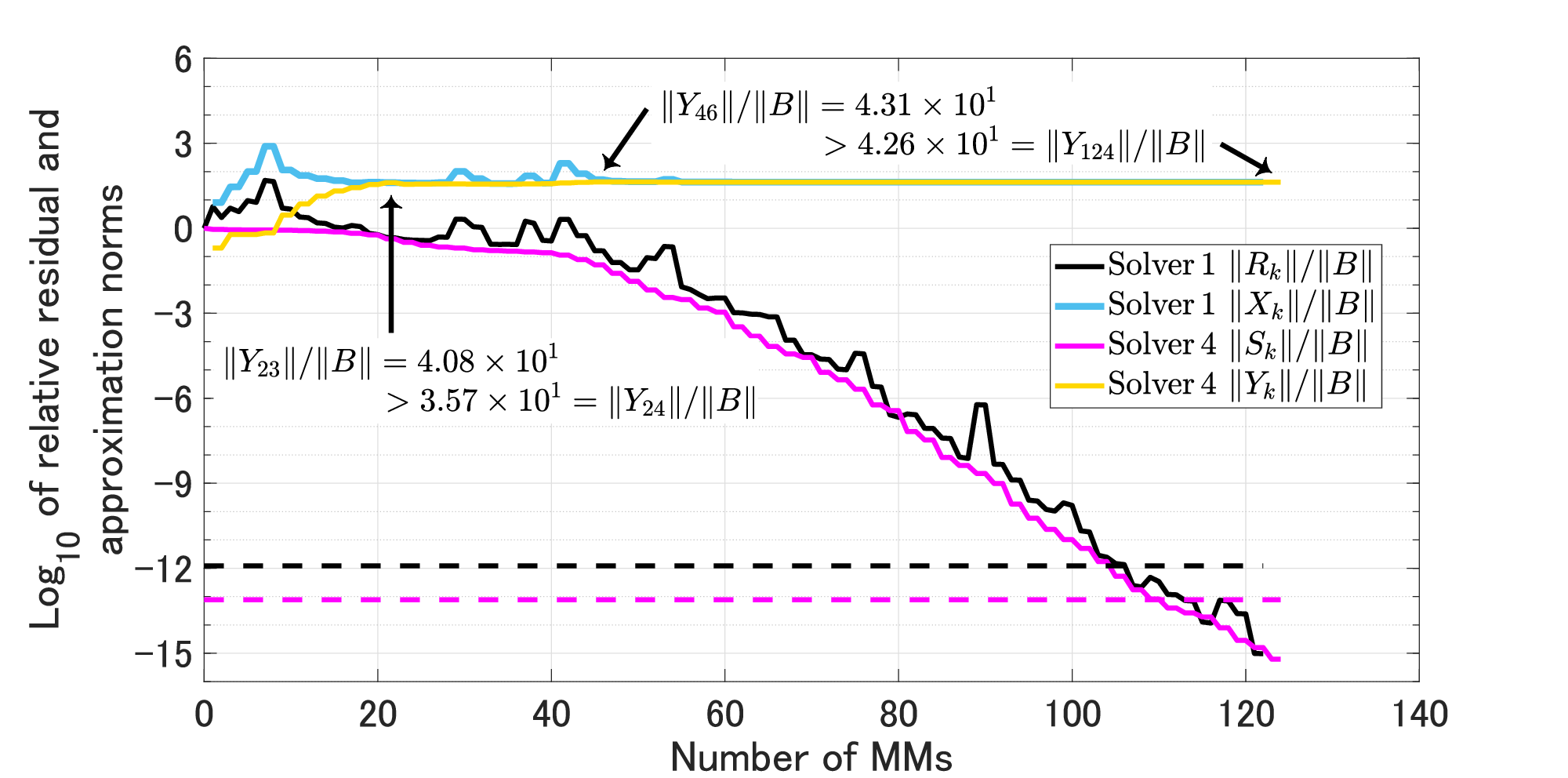}
		\caption{Convergence histories of the relative residual and approximation norms of Solvers~1 and 4 for bfwa782 with $s=16$.}\label{Ex2_2}
\end{figure}

\subsubsection{Comparison between Bl-CIRS and Gl-CIRS}

We finally compare the convergence behaviors of Solvers~4 and 5 to investigate the difference between Bl-CIRS and Gl-CIRS. 
\Cref{Ex3} shows the convergence histories of the relative residual norms of Solvers~4 and 5 for all the combinations of the test matrices and the number of $s$. 
The number of iterations and $\log_{10}$ of the relative residual norm are plotted on the horizontal and vertical axes, respectively. 

From \Cref{Results} and \Cref{Ex3}, we observe the following. 
The residual gaps of Solvers~4 and 5 are sufficiently small and their approximations attain the same level of accuracy. 
However, Solver~4 exhibits even smoother convergence behavior than Solver~5, and the convergence speed of Solver~4 is slightly faster in several cases. 
\Cref{Ex3_3} shows that Solver~4 does not always offer smaller residual norms than Solver~5. 

As seen in the comparisons between the global- and block-type solvers, Gl-CIRS and Bl-CIRS have their strengths. 
For instance, Gl-CIRS requires smaller computational costs per iteration than Bl-CIRS, but Bl-CIRS could be more efficient in parallel. 
Bl-CIRS excels in smoothing residual norms as shown above, while Gl-CIRS would be easier to apply to more general matrix equations (cf.~\cite{AiharaImakuraMorikuni2022SIMAX}). 
Therefore, it is difficult to state which between Gl-CIRS and Bl-CIRS is better. 

\begin{figure}[t]
	\begin{minipage}{0.49\textwidth}
		\centering
		\includegraphics[scale=0.2]{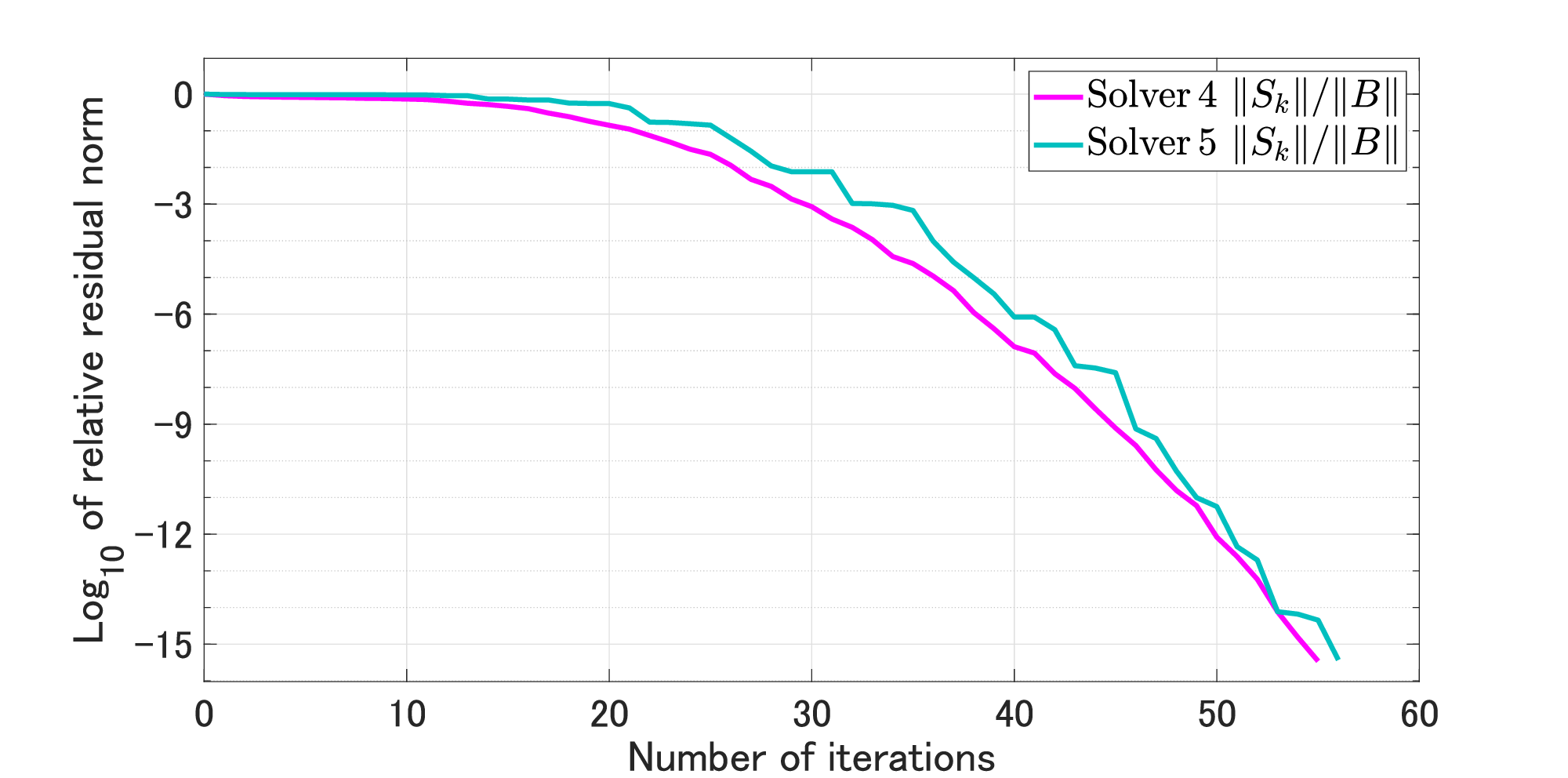}
		\subcaption{cdde2 with $s=16$.}\label{Ex3_1}
	\end{minipage}
	\begin{minipage}{0.49\textwidth}
		\centering
		\includegraphics[scale=0.2]{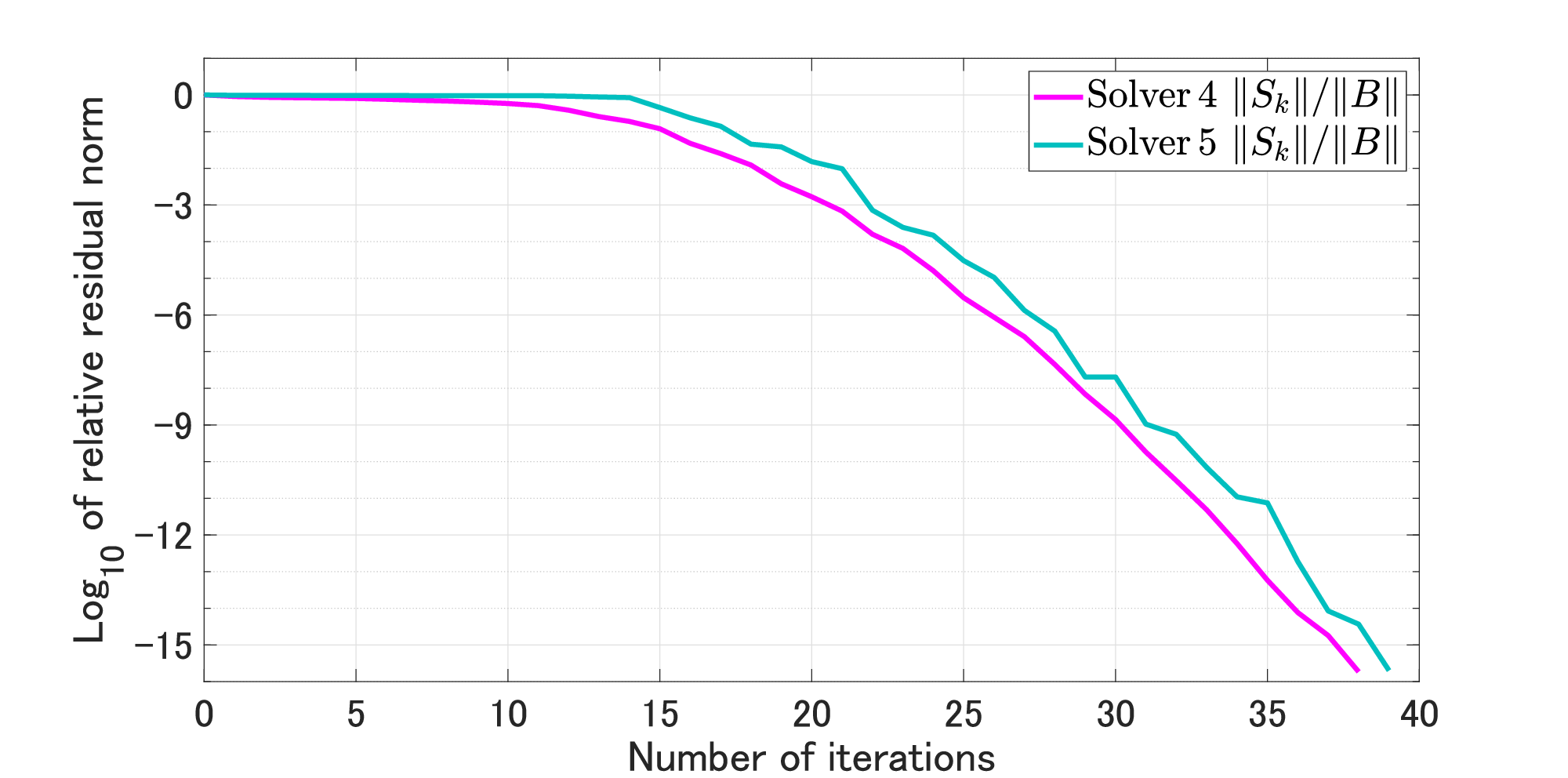}
		\subcaption{cdde2 with $s=32$.}\label{Ex3_2}
	\end{minipage}
	\\
	\begin{minipage}{0.49\textwidth}
		\centering
		\includegraphics[scale=0.2]{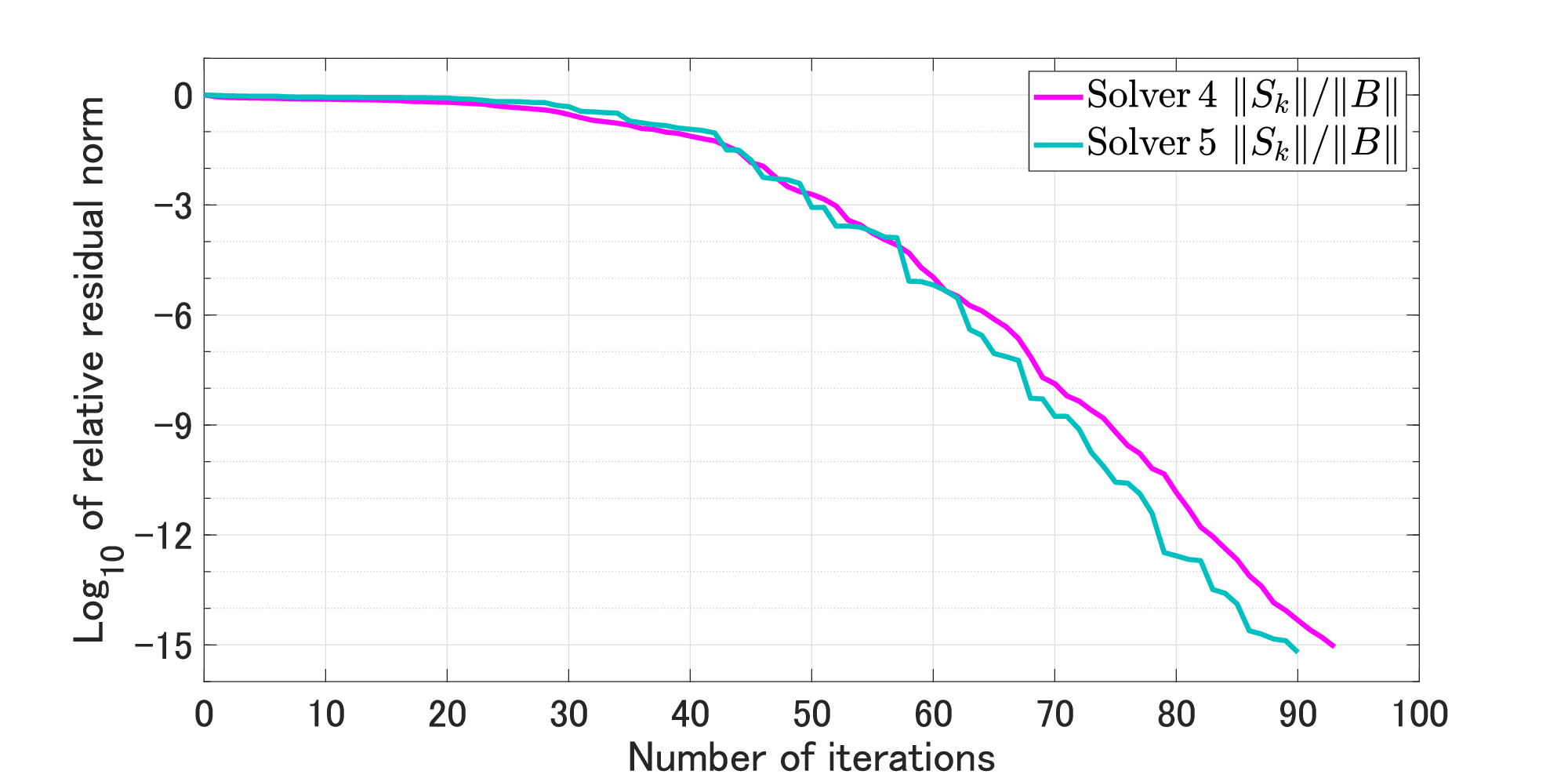}
		\subcaption{pde2961 with $s=16$.}\label{Ex3_3}
	\end{minipage}
	\begin{minipage}{0.49\textwidth}
		\centering
		\includegraphics[scale=0.2]{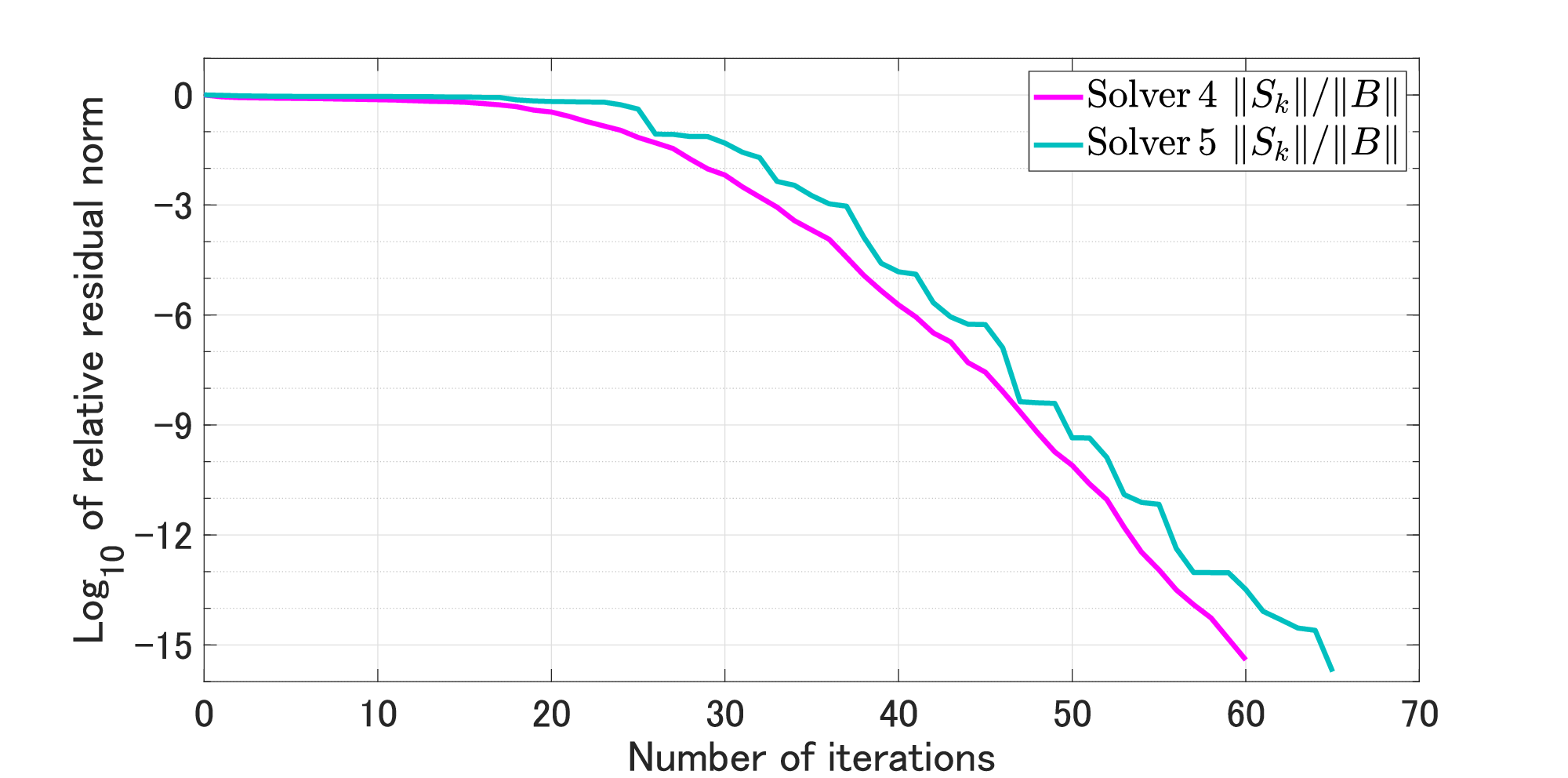}
		\subcaption{pde2961 with $s=32$.}\label{Ex3_4}
	\end{minipage}
	\\
	\begin{minipage}{0.49\textwidth}
		\centering
		\includegraphics[scale=0.2]{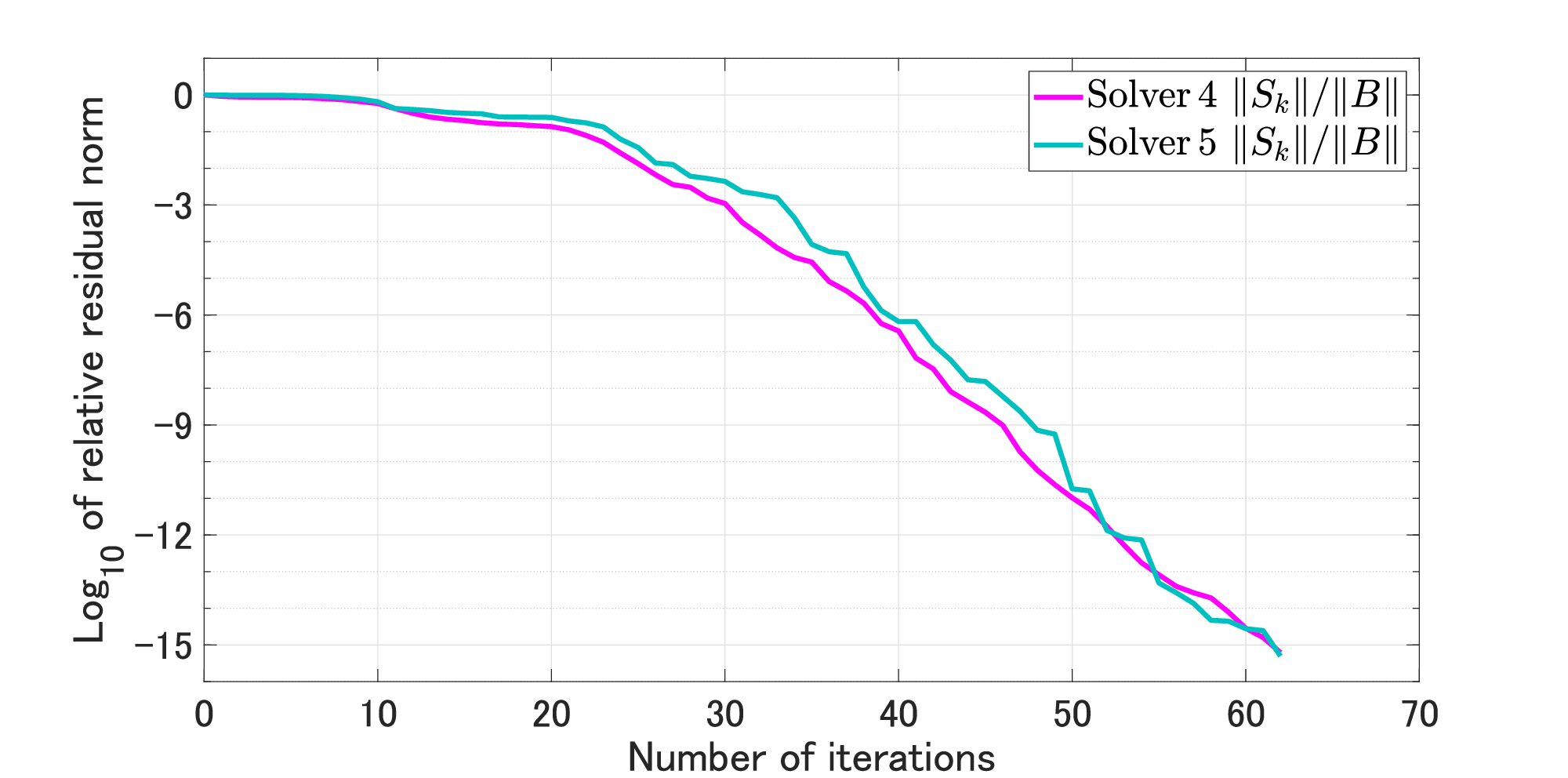}
		\subcaption{bfwa782 with $s=16$.}\label{Ex3_5}
	\end{minipage}
	\begin{minipage}{0.49\textwidth}
		\centering
		\includegraphics[scale=0.2]{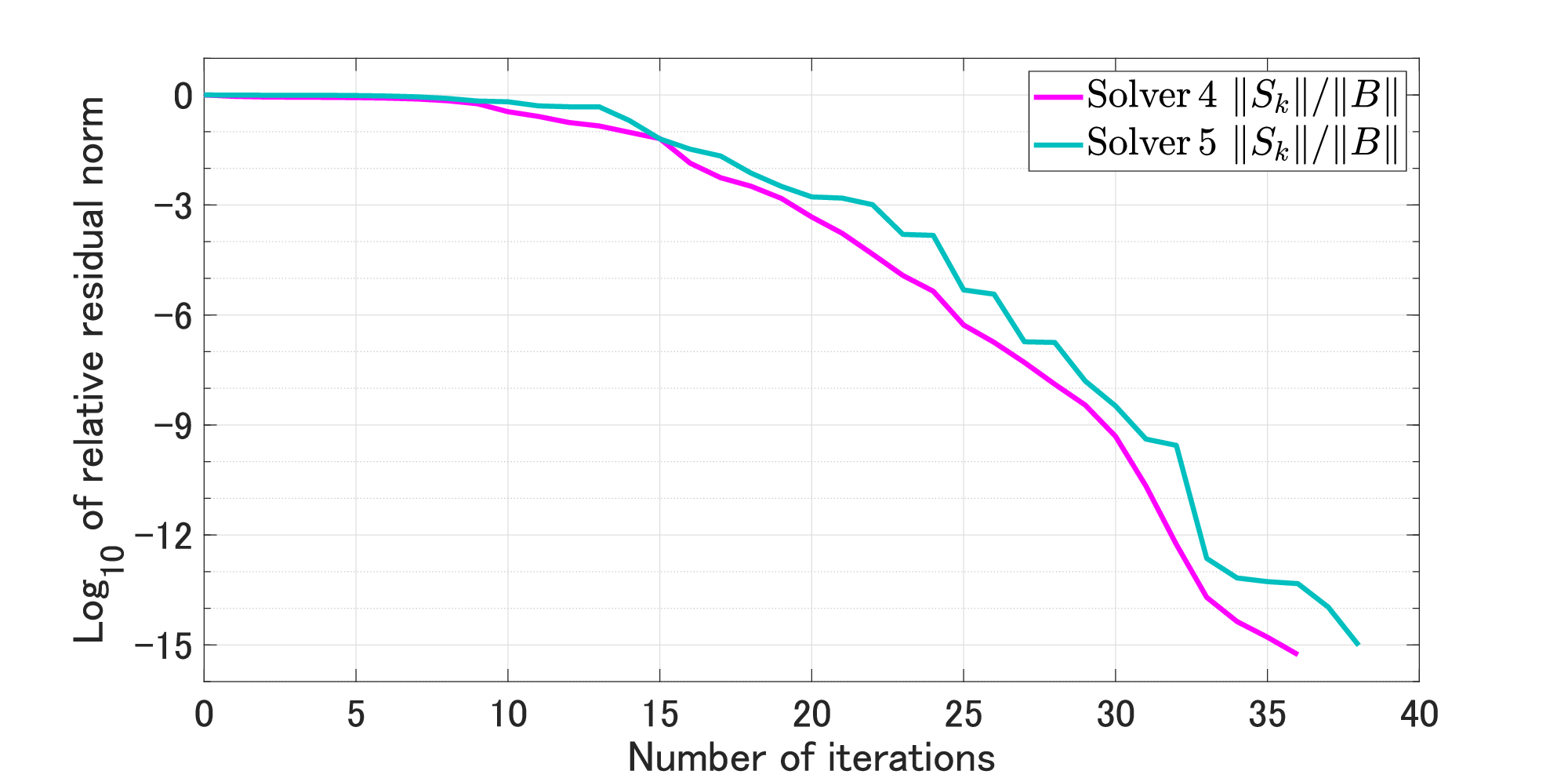}
		\subcaption{bfwa782 with $s=32$.}\label{Ex3_6}
	\end{minipage}
	\caption{Convergence histories of Solvers~4 and 5 for test matrices.}\label{Ex3}
\end{figure}

\section{Rounding error analysis for the residual gap}\label{sec4}

In this section, a detailed rounding error analysis shows that Bl-CIRS with the orthonormalization strategy suppresses the residual gap. 
The new results evaluate the residual gap fully in finite precision arithmetic and complement the results in \cite{AiharaImakuraMorikuni2022SIMAX} when using backward stable orthonormalization processes; Givens rotations and Householder transformations. 
A more general case is also presented in \Cref{Appendix}. 

We first describe our main result on the residual gap; the proof will be presented in \cref{sec4.3}. 
\begin{theorem}\label{add_th}
In finite precision arithmetic, let $X_k \in \mathbb{F}^{n \times s}$ and $R_k \in \mathbb{F}^{n \times s}$ be the $k$th approximation and residual, respectively, generated by the recursion formulas 
\begin{align}
X_k = X_{k-1} + \widehat{Q}_{k-1} \alpha_{k-1}^\square, \quad R_{k} = R_{k-1} - (A\widehat{Q}_{k-1}) \alpha_{k-1}^\square, \label{eq:recur_form}
\end{align}
where $\widehat{Q}_{k-1}$ is a Q-factor of the direction matrix $P_{k-1}$ obtained from the QR decomposition with Givens rotations or Householder transformations. 
Then, with $X_0 = O$, the norm of the residual gap~$G_{R_k} := (B-AX_k) - R_k$ is bounded as follows: 
\begin{align}
\| G_{R_k} \| < \left(8 \sqrt{s} \gamma_{m+3s} + \gamma_1\right) k \| A \| \max_{0<i\leq k} \| X_i \|  + k \gamma_1 \max_{0< i\leq k} \| R_i \|. \label{eq:G_Rk}
\end{align}
Here, we refer to the subsections below for the values of $\gamma$ and the related assumptions. 
\end{theorem}

This theorem shows that we must avoid large peaks of the approximation and residual norms to suppress the increase of the residual gap. 
For Bl-CIRS, replacing $X_k$ and $R_k$ by $Y_k$ and $S_k$, respectively, in \eqref{eq:recur_form}, \Cref{add_th} holds for the residual gap $G_{S_k} := (B-AY_k) - S_k$. 
Therefore, as indicated right after \cite[Corollary~2.3]{AiharaImakuraMorikuni2022SIMAX}, even when using an inexact orthonormalization, Bl-CIRS smoothes the convergence of the residual and approximation norms and can reduce the residual gap. 
This theoretical result is consistent with the numerical results in \cref{sec3}. 

To prove \Cref{add_th}, we exploit rounding error models in \cite{Higham2002} and evaluate local error bounds for \eqref{eq:recur_form}.

\subsection{Error models for matrix computations}

Our analysis is based on the following models for matrix operations in finite precision arithmetic \cite{Higham2002}: 
\begin{align}
	& \mathrm{fl}(P+Q) =  P + Q + E',\quad \|E'\| \leq \textbf{u} \|P  + Q\|, \label{eq:fl_matrix_sum} \\
	& \mathrm{fl}(P \alpha^\square) = P \alpha^\square + E^\square,\quad \|E^\square\| \leq \gamma_s \|P\| \|\alpha^\square\|, \label{eq:fl_model3}\\
	& \mathrm{fl}(AP) = AP + E''', \quad \|E'''\| \leq \gamma_m \| A \| \| P \| \label{eq:fl_model4}
\end{align}
for given $P, Q \in \mathbb{F}^{n\times s}$ and $\alpha^\square \in \mathbb{F}^{s\times s}$, where $\gamma_k$ is defined by 
\begin{align*}
	\gamma_{k} := \frac{k \mathbf{u}}{1-k \mathbf{u}},
\end{align*}
and an assumption $k\mathbf{u} < 1$ is implicitly imposed whenever we use $\gamma_k$. 
For the notations $\mathbb{F}$, fl($\cdot $), $\textbf{u}$, and $m$, we refer the reader to the sentence right before \Cref{Th1}. 
In the following, we focus on using Givens rotation and Householder transformation to orthonormalize the columns of $P_{k-1} \in \mathbb{R}^{n\times s}$ and introduce their known rounding error bounds. 

The QR decomposition of $P_{k-1} \in \mathbb{R}^{n \times s}$ using Given rotations~\cite[section~19.6]{Higham2002} or Householder transformations~\cite[section~19.3]{Higham2002} numerically computes a Q-factor $\widehat{Q}_{k-1}^{\mathrm{full}} = \left[ \widehat{Q}_{k-1} \ \widehat{Q}_{k-1}' \right] \in \mathbb{R}^{n \times n}$ with $\widehat{Q}_{k-1} \in \mathbb{R}^{n \times s}$ such that
\begin{align}
\begin{split}
&\widehat{Q}_{k-1} = Q_{k-1}
\left(
\begin{bmatrix}
	I_s \\ 
	O 
\end{bmatrix}
+ \Delta I_{k-1} \right),\quad \Delta I_{k-1} \in \mathbb{R}^{n\times s},\\
&\| \Delta I_{k-1}(:, j) \|_{2} \leq 
\begin{cases}
	\tilde{\gamma}_{n+s-2} & \text{for Givens}, \\
	\tilde{\gamma}_{ns} & \text{for Householder},
\end{cases} 
\end{split}
\label{eq:Q}
\end{align}
for $j = 1, 2, \dots, s$, where $Q_{k-1} \in \mathbb{R}^{n \times n}$ is the product of Givens rotations or Householder transformations and is the Q-factor (an exactly orthogonal matrix) of the full QR decomposition of $P_{k-1}$ and $\|\cdot \|_2$ denotes the Euclidean norm. Here, the notation $X(:, j)$ follows the MATLAB convention, that is, it stands for the $j$th column of the matrix $X$. 
The scalar $\tilde \gamma_k$ is defined by 
\begin{align*}
\tilde{\gamma}_{k} = \frac{c k \mathbf{u}}{1 - c k \mathbf{u}}, 
\end{align*}
where $c$ is a small integer constant greater than one \cite[eq.~(3.8)]{Higham2002}, and $ck\mathbf{u} < 1$ is assumed similarly to $\gamma_k$. 
From \eqref{eq:Q}, it holds that 
\begin{align*}
\left\| \widehat{Q}_{k-1} \right\| \leq \| I_s \| + \left\| \Delta I_{k-1} \right\| \leq 
\begin{cases}
	\sqrt{s} (1+ \tilde{\gamma}_{n+s-2}) & \text{for Givens}, \\	
	\sqrt{s} (1+ \tilde{\gamma}_{ns}) & \text{for Householder}.	
\end{cases}
\end{align*}

In the analysis below, we will use the bounds
\begin{align}
	\tilde{\gamma}_j + \tilde{\gamma}_k + \tilde{\gamma}_j \tilde{\gamma}_k & < \tilde{\gamma}_{j+k} && \text{for} ~ c(j+k)\mathbf{u} < 1, \label{eq:bound_gamma_4} \\
	\gamma_k & < \tilde{\gamma}_k && \text{for} ~ ck\mathbf{u} < 1, \label{eq:bound_gamma_gamma1} \\
	\gamma_j \tilde{\gamma}_k & < \min (\gamma_j, \tilde{\gamma}_k) \leq \max (\gamma_j, \tilde{\gamma}_k) && \text{for} ~ \max(j,k) c\mathbf{u} < 1/2 \label{eq:bound_gamma_gamma2}
\end{align}
for positive integers $j$ and $k$, and a constant $c$; see also \cite[Lemma~3.3]{Higham2002}.

\subsection{Local errors in recursion formulas}

Here, we analyze local errors when using the recursion formulas \eqref{eq:recur_form}. 
The local errors $E_{X_k}$ and $E_{R_k}$ in the updated approximation and residual, respectively, can be expressed as follows: 
\begin{align}
	X_k & = \mathrm{fl}(X_{k-1} + \mathrm{fl}(\widehat{Q}_{k-1} \alpha_{k-1}^\square)) \notag \\
	& = X_{k-1} + \mathrm{fl}(\widehat{Q}_{k-1} \alpha_{k-1}^\square) + E_{X_k}' \label{eq:recur_E_Xk'} \\
	& = X_{k-1} + \widehat{Q}_{k-1} \alpha_{k-1}^\square + E_{X_k}' + E_{X_k}^\square \label{eq:recur_E_Xk_sq} \\
	& = X_{k-1} + \widehat{Q}_{k-1} \alpha_{k-1}^\square + E_{X_k}, \notag 
\end{align}
where $E_{X_k} := E_{X_k}' + E_{X_k}^\square$, and
\begin{align*}
	R_k 
	& = \mathrm{fl}(R_{k-1} - \mathrm{fl}(\mathrm{fl}(A \widehat{Q}_{k-1})\alpha_{k-1}^\square)) \\
	& = R_{k-1} - \mathrm{fl}(\mathrm{fl}(A \widehat{Q}_{k-1})\alpha_{k-1}^\square) - E_{R_k}' \\
	& = R_{k-1} - \mathrm{fl}(A \widehat{Q}_{k-1})\alpha_{k-1}^\square - E_{R_k}' - E_{R_k}^\square \\
	& = R_{k-1} - A \widehat{Q}_{k-1} \alpha_{k-1}^\square - E_{R_k}' - E_{R_k}^\square - E_{R_k}''' \alpha_{k-1}^\square \\
	& = R_{k-1} - A \widehat{Q}_{k-1} \alpha_{k-1}^\square - E_{R_k},
\end{align*}
where $E_{R_k} := E_{R_k}' + E_{R_k}^\square + E_{R_k}''' \alpha_{k-1}^\square$. 
Subsequently, the local errors satisfy the bounds below.

\begin{lemma} \label{lm:local_error}
For \eqref{eq:recur_form} in finite precision arithmetic, assume that $\widehat{Q}_{k-1}$ is a numerically computed Q-factor of the QR decomposition of $P_{k-1}$ using Givens rotations or Householder transformations, and that $c$, $n$, $s$, and $\mathbf{u}$ satisfy 
\begin{align}
\begin{cases}
c \left[ (n+2(s-1)) (\sqrt{s}+1) + 1 \right]\mathbf{u} < \frac{1}{2} \quad & \text{for Givens}, \\
c \left[(n+1) (\sqrt{s}+1) s + 1\right] \mathbf{u} < \frac{1}{2} & \text{for Householder}.	
\end{cases}\label{add_3}
\end{align}
Then, for $k \geq 1$, the local errors are bounded as follows: 
\begin{align}
	\| E_{X_k} \| & < 4 \sqrt{s} \gamma_s ( \| X_{k-1} \| +  \| X_k \| ) + \gamma_1 \| X_k \|, \label{eq:E_Xk}  \\
	\| E_{R_k} \| & < \gamma_1 \| R_k \| + 4 \sqrt{s} (\gamma_m + 2 \gamma_s) \| A \| \left( \| X_{k-1} \| + \| X_k \| \right). \label{eq:E_Rk}
\end{align}
\end{lemma}
\begin{proof}
The assumptions imply that $c(n+s-2)\mathbf{u}<1$ and $c n s \mathbf{u} < 1$ hold and ensure the existence of $\widehat{Q}_{k-1}$ such that \eqref{eq:Q} is satisfied.

From \eqref{eq:fl_matrix_sum}, it follows that 
\begin{align*}
\| E_{X_k}' \| \leq \mathbf{u} \| X_{k-1} + \mathrm{fl}(\widehat{Q}_{k-1} \alpha_{k-1}^\square) \| \leq \mathbf{u} (\| X_k \| + \| E_{X_k}' \|),
\end{align*}
in which the last inequality is from \eqref{eq:recur_E_Xk'}, that is, $X_{k-1} + \mathrm{fl}(\widehat{Q}_{k-1} \alpha_{k-1}^\square) = X_k - E_{X_k}'$.
Subsequently, we obtain the bound
\begin{align}
	\| E_{X_k}' \| \leq \gamma_1 \| X_k \|. \label{eq:E_Xk'} 
\end{align}
Next, it follows from \eqref{eq:fl_model3} that 
\begin{align}
	\| E_{X_k}^\square \|
	\leq \gamma_s \| \widehat{Q}_{k-1} \| \| \alpha_{k-1}^\square \| 
	\leq 
	\begin{cases}
		\sqrt{s} \left(1 + \tilde{\gamma}_{n+s-2} \right) \gamma_s \| \alpha_{k-1}^\square \| & \text{for Givens}, \\
		\sqrt{s} \left(1 + \tilde{\gamma}_{ns} \right) \gamma_s \| \alpha_{k-1}^\square \| & \text{for Householder}. 
	\end{cases}	\label{new_add_1}
\end{align}
Combining these bounds with \eqref{eq:recur_E_Xk_sq} and \eqref{eq:E_Xk'} gives
\begin{align*}
	\| \widehat{Q}_{k-1} \alpha_{k-1}^\square \| 
	& \leq \| X_{k-1} \| + \| X_k \| + \| E_{X_k}' \| + \| E_{X_k}^\square \| \\
	& \leq 
	\begin{cases}
		\| X_{k-1} \| + (1+\gamma_1) \| X_k \| + \sqrt{s} \left( 1 + \tilde{\gamma}_{n+s-2} \right) \gamma_s \| \alpha_{k-1}^\square \| & \text{for Givens}, \\
		\| X_{k-1} \| + (1+\gamma_1) \| X_k \| + \sqrt{s} \left( 1 + \tilde{\gamma}_{ns} \right) \gamma_s \| \alpha_{k-1}^\square \| & \text{for Householder}. 
	\end{cases}
\end{align*}
Regarding $\alpha_{k-1}^\square$, we have 
\begin{align*}
	\| \alpha_{k-1}^\square \|
	& = \| Q_{k-1}(:,1\!:\!s)\, \alpha_{k-1}^\square \| \\
	& = \left\| \left( \widehat{Q}_{k-1} - Q_{k-1} 
	\Delta I_{k-1} \right) \alpha_{k-1}^\square \right\| \\
	& \leq \| \widehat{Q}_{k-1} \alpha_{k-1}^\square \| + \left\| 
	\Delta I_{k-1} \alpha_{k-1}^\square \right\| \\ 
	& \leq 
	\begin{cases}
		\| \widehat{Q}_{k-1} \alpha_{k-1}^\square \| + \sqrt{s} \tilde{\gamma}_{n+s-2} \| \alpha_{k-1}^\square \| & \text{for Givens}, \\
		\| \widehat{Q}_{k-1} \alpha_{k-1}^\square \| + \sqrt{s} \tilde{\gamma}_{ns} \| \alpha_{k-1}^\square \| & \text{for Householder}.
	\end{cases}
\end{align*}
Thus, we obtain the bounds 
\begin{align*}
	\| \alpha_{k-1}^\square \| 
	& \leq 
	\begin{cases}
		\| X_{k-1} \| + (1+\gamma_1) \| X_k \| + \sqrt{s} ( (1 + \tilde{\gamma}_{n+s-2}) \gamma_s + \tilde{\gamma}_{n+s-2}) \| \alpha_{k-1}^\square \|, \\
		\| X_{k-1} \| + (1+\gamma_1) \| X_k \| + \sqrt{s} ( (1 + \tilde{\gamma}_{ns}) \gamma_s + \tilde{\gamma}_{ns}) \| \alpha_{k-1}^\square \|
	\end{cases}	\\
	& < 
	\begin{cases}
		\| X_{k-1} \| + (1+\gamma_1) \| X_k \| +  \sqrt{s} \tilde{\gamma}_{n+2(s-1)} \| \alpha_{k-1}^\square \| & \text{for Givens}, \\
		\| X_{k-1} \| + (1+\gamma_1) \| X_k \| +  \sqrt{s} \tilde{\gamma}_{(n+1)s} \| \alpha_{k-1}^\square \| & \text{for Householder}. 
	\end{cases}
\end{align*}
Here, the second inequality is from 
\begin{align*}
	(1+\tilde{\gamma}_{n+s-2}) \gamma_s + \tilde{\gamma}_{n+s-2} 
	& < \tilde{\gamma}_{s} + \tilde{\gamma}_{n+s-2} + \tilde{\gamma}_s \tilde{\gamma}_{n+s-2} \quad (\because \eqref{eq:bound_gamma_gamma1}) \\
	& < \tilde{\gamma}_{n+2(s-1)} \quad (\because \eqref{eq:bound_gamma_4})
\end{align*}
and
\begin{align*}
	(1+\tilde{\gamma}_{ns}) \gamma_s + \tilde{\gamma}_{ns} 
	& < \tilde{\gamma}_{s} + \tilde{\gamma}_{ns} + \tilde{\gamma}_s \tilde{\gamma}_{ns} \quad (\because \eqref{eq:bound_gamma_gamma1}) \\
	& < \tilde{\gamma}_{(n+1)s} \quad (\because \eqref{eq:bound_gamma_4}) 
\end{align*}
hold for Givens and Householder, respectively. 
Rearranging for $\alpha_{k-1}^\square$, we have
\begin{align*}
	\| \alpha_{k-1}^\square \| < 
	\begin{cases}
		\dfrac{1}{1-\sqrt{s}\tilde{\gamma}_{n+2(s-1)}} \left[ \| X_{k-1} \| + (1+\gamma_1) \| X_k \| \right] & \text{for Givens}, \\
		\dfrac{1}{1-\sqrt{s}\tilde{\gamma}_{(n+1)s}} \left[ \| X_{k-1} \| + (1+\gamma_1) \| X_k \| \right] & \text{for Householder}.
	\end{cases}
\end{align*}
Here, $1-\sqrt{s}\tilde{\gamma}_{n+2(s-1)} > 0$ and $1 - \sqrt{s} \tilde{\gamma}_{(n+1)s} > 0$ are equivalent to 
\begin{align*}
c (n+2(s-1)) (\sqrt{s}+1) \mathbf{u} < 1\quad \text{and}\quad c (n+1) s (\sqrt{s}+1) \mathbf{u} < 1, 
\end{align*}
respectively, and these are satisfied by the assumptions \eqref{add_3}. 
Because the coefficients satisfy 
\begin{align*}
\frac{1}{1-\sqrt{s}\tilde{\gamma}_{n+2(s-1)}} = 1 + \frac{c (n+2(s-1)) \sqrt{s} \mathbf{u}}{1-c(n+2(s-1)) (\sqrt{s}+1) \mathbf{u}} < 1 + \tilde{\gamma}_{(n+2(s-1))(\sqrt{s}+1)}
\end{align*}
and
\begin{align*}
\frac{1}{1-\sqrt{s}\tilde{\gamma}_{(n+1)s}} = 1 + \frac{c (n+1) s^{3/2} \mathbf{u}}{1-c(n+1)(\sqrt{s}+1)s \mathbf{u}} < 1 + \tilde{\gamma}_{(n+1)(\sqrt{s}+1)s},
\end{align*}
we obtain the bounds 
\begin{align*}
	\| \alpha_{k-1}^\square \|
	& < 
	\begin{cases}
		(1 + \tilde{\gamma}_{(n+2(s-1))(\sqrt{s}+1)}) \left[ \| X_{k-1} \| + (1+\gamma_1) \| X_k \| \right] & \text{for Givens}, \\
		(1 + \tilde{\gamma}_{(n+1)(\sqrt{s}+1)s}) \left[ \| X_{k-1} \| + (1+\gamma_1) \| X_k \| \right] & \text{for Householder}.
	\end{cases}
\end{align*}
Noting that 
\begin{align*}
	(1 + \tilde{\gamma}_{(n+2(s-1)) (\sqrt{s}+1)}) (1+\gamma_1) 
	&< 1 + \tilde{\gamma}_1 + \tilde{\gamma}_{(n+2(s-1)) (\sqrt{s}+1)} + \tilde{\gamma}_1 \tilde{\gamma}_{(n+2(s-1)) (\sqrt{s}+1)} \quad (\because \eqref{eq:bound_gamma_gamma1}) \\
	&< 1 + \tilde{\gamma}_{(n+2(s-1)) (\sqrt{s}+1)+1} \quad (\because \eqref{eq:bound_gamma_4})
\end{align*}
and
\begin{align*}
	(1 + \tilde{\gamma}_{(n+1) (\sqrt{s}+1) s}) (1+\gamma_1) 
	& < 1 + \tilde{\gamma}_1 + \tilde{\gamma}_{(n+1)(\sqrt{s}+1)s} + \tilde{\gamma}_1 \tilde{\gamma}_{(n+1)(\sqrt{s}+1)s} \quad (\because \eqref{eq:bound_gamma_gamma1}) \\
	& < 1 + \tilde{\gamma}_{(n+1)(\sqrt{s}+1)s+1}, \quad (\because \eqref{eq:bound_gamma_4})
\end{align*}
we obtain the final evaluations for $\alpha_{k-1}^{\square}$ as follows: 
\begin{align}
	\| \alpha_{k-1}^\square \| <
	\begin{cases}
		(1+\tilde{\gamma}_{(n+2(s-1)) (\sqrt{s}+1)}) \| X_{k-1} \| \\ 
		\quad +  (1 + \tilde{\gamma}_{(n+2(s-1)) (\sqrt{s}+1)+1}) \| X_k \| & \text{for Givens}, \\
		(1+\tilde{\gamma}_{(n+1) (\sqrt{s}+1) s}) \| X_{k-1} \| \\
		\quad +  (1 + \tilde{\gamma}_{(n+1) (\sqrt{s}+1) s + 1}) \| X_k \| & \text{for Householder}. 
	\end{cases} \label{eq:bound_alpha}
\end{align}
Now, to evaluate $\| E_{X_k}^\square \|$ simply, we use the bounds
\begin{align*}
	&\gamma_s(1+\tilde{\gamma}_{(n+2(s-1)) (\sqrt{s}+1)+1}) = \gamma_s + \gamma_s\tilde{\gamma}_{(n+2(s-1)) (\sqrt{s}+1)+1} < \gamma_s + \gamma_s = 2 \gamma_s,\\ 
	&\gamma_s(1+\tilde{\gamma}_{(n+1)(\sqrt{s}+1)s+1}) = \gamma_s + \gamma_s\tilde{\gamma}_{(n+1)(\sqrt{s}+1)s+1} < \gamma_s + \gamma_s = 2 \gamma_s,
\end{align*}
where the inequalities are from \eqref{eq:bound_gamma_gamma2}, and the corresponding assumptions \eqref{add_3} are imposed for Givens and Householder, respectively. 
Following a similar argument to the above bounds, substituting \eqref{eq:bound_alpha} into the bounds \eqref{new_add_1} of $\| E_{X_k}^\square \|$ gives
\begin{align}
\begin{split}
	\| E_{X_k}^\square \| 
	& < 
	\begin{cases}
	2 \sqrt{s} \left(1 + \tilde{\gamma}_{n+s-2} \right) \gamma_s (\| X_{k-1} \| + \| X_k \| ) & \text{for Givens}, \\
	2 \sqrt{s} \left(1 + \tilde{\gamma}_{ns} \right) \gamma_s (\| X_{k-1} \| + \| X_k \|) & \text{for Householder} \\
	\end{cases}	\\
	& < 4 \sqrt{s} \gamma_s ( \| X_{k-1} \| +  \| X_k \| ). \label{eq:E_Xk_sq}
\end{split}
\end{align}
Summing up \eqref{eq:E_Xk'} and \eqref{eq:E_Xk_sq}, we obtain the bound~\eqref{eq:E_Xk}.

Next, it holds that $\| E_{R_k}' \| \leq \gamma_1 \| R_k \|$ as in \eqref{eq:E_Xk'}. 
From \eqref{eq:fl_model3}, \eqref{eq:fl_model4}, \eqref{eq:bound_gamma_gamma2}, and \eqref{eq:bound_alpha}, similarly to \eqref{eq:E_Xk_sq}, we obtain the following bound: 
\begin{align*}
	\| E_{R_k}^\square \| 
	& \leq \gamma_s \| \mathrm{fl}(A \widehat{Q}_{k-1}) \| \| \alpha_{k-1}^\square \| \\
	& < 
	\begin{cases}
		\begin{aligned}
			& \gamma_s (\| A \widehat{Q}_{k-1} \| + \gamma_m \| A \| \| \widehat{Q}_{k-1} \|) [ (1 + \tilde{\gamma}_{(n+2(s-1))(\sqrt{s}+1)}) \| X_{k-1} \| \\
			& \quad + (1 + \tilde{\gamma}_{(n+2(s-1))(\sqrt{s}+1)+1}) \| X_k \| ]
		\end{aligned}
		& \text{for Givens}, \\
		\begin{aligned}
			& \gamma_s (\| A \widehat{Q}_{k-1} \| + \gamma_m \| A \| \| \widehat{Q}_{k-1} \|) [ (1 + \tilde{\gamma}_{(n+1)(\sqrt{s}+1)s}) \| X_{k-1} \| \\
			& \quad + (1 + \tilde{\gamma}_{(n+1)(\sqrt{s}+1)s+1}) \| X_k \| ] 
		\end{aligned} 
		& \text{for Householder}
	\end{cases}	\\
	& < 2 \gamma_s (1+\gamma_m) \| A \| \| \widehat{Q}_{k-1} \| ( \| X_{k-1} \| + \| X_k \| ) \\ 
	& < 
	\begin{cases}
		4 \sqrt{s} \gamma_s (1+\tilde{\gamma}_{n+s-2}) \| A \| ( \| X_{k-1} \| + \| X_k \| )
		& \text{for Givens}, \\
		4 \sqrt{s} \gamma_s (1+\tilde{\gamma}_{ns}) \| A \| ( \| X_{k-1} \| + \| X_k \| )
		& \text{for Householder}
	\end{cases} \\
	& < 8 \sqrt{s} \gamma_s \| A \| ( \| X_{k-1} \| + \| X_k \| ). 
\end{align*}
From \eqref{eq:fl_model4}, in a similar fashion to the bounds above, we have
\begin{align*}
	\| E_{R_k}''' \alpha_{k-1}^\square \| 
	& \leq \gamma_m \| A \| \| \widehat{Q}_{k-1} \| \|\alpha_{k-1}^\square\| \\
	& < 2 \gamma_m \| A \| \| \widehat{Q}_{k-1} \| ( \| X_{k-1} \| + \| X_k \| ) \\
	& < 4 \sqrt{s} \gamma_m \| A \| ( \| X_{k-1} \| + \| X_k \| ). 
\end{align*}
Hence, summing up the bounds for $\|E'_{R_k}\|$, $\| E_{R_k}^\square \|$, and $\| E_{R_k}''' \alpha_{k-1}^\square\|$ gives \eqref{eq:E_Rk}.
\end{proof}

\subsection{Bound for the residual gap}\label{sec4.3}

We provide a proof of \Cref{add_th} under the assumptions in \Cref{lm:local_error}. 

\begin{proof}[Proof of \Cref{add_th}]
From \Cref{lm:local_error}, we have that 
\begin{align*}
\| G_{R_k} \| 
& = \| [B - A (X_{k-1} + \widehat{Q}_{k-1} \alpha_{k-1}^\square + E_{X_k})] - (R_{k-1} - A \widehat{Q}_{k-1} \alpha_{k-1}^\square - E_{R_k}) \| \\
& \leq \| (B - A X_{k-1}) - R_{k-1} \| + \| A \| \| E_{X_k} \| + \| E_{R_k} \| \\
& \leq \| A \| \sum_{i=1}^k \| E_{X_i} \| + \sum_{i=1}^k \| E_{R_{i}} \| \\
& < 4 \sqrt{s} \gamma_s \| A \| \sum_{i=1}^k (\| X_{i-1} \| + \| X_i \|) + \gamma_1 \| A \| \sum_{i=1}^k \| X_i \| \\
& \qquad + \gamma_1 \sum_{i=1}^k \| R_i \| + 4 \sqrt{s} (\gamma_m + 2 \gamma_s) \| A \| \sum_{i=1}^k \left( \| X_{i-1} \| + \| X_i \| \right) \\
& < \left[8 \sqrt{s} (\gamma_m + 3 \gamma_s) + \gamma_1\right] \| A \| \sum_{i=1}^k \| X_i \| + \gamma_1 \sum_{i=1}^k \| R_i \| \\
& < ( 8 \sqrt{s} \gamma_{m+3s} + \gamma_1 ) k \| A \| \max_{0<i\leq k} \| X_i \| + k \gamma_1 \max_{0<i\leq k} \| R_i \|.
\end{align*}
\end{proof}

\section{Concluding remarks}\label{sec5}

In this study, we have proposed a block version of the cross-interactive residual smoothing (\Cref{alg2} referred to as Bl-CIRS) incorporated with Lanczos-type solvers for linear systems with multiple right-hand sides. 
The proposed smoothing scheme can obtain a monotonically decreasing sequence of the residual norms along with a smoothly increasing sequence of the associated approximation norms, thereby reducing the residual gap and improving the attainable accuracy of the approximations. 
The key point of the block version is to define the smoothing parameter as an $s$-by-$s$ matrix and orthonormalize the columns of $n$-by-$s$ auxiliary matrices used to update the approximations and residuals. 
We have demonstrated the performance of the proposed scheme through numerical experiments in \cref{sec3}. 
The rounding error analysis in \Cref{add_th} shows that Bl-CIRS is useful in avoiding a large residual gap, even when allowing inexactness of the orthonormalization, which is weaker than our previous assumption made in \cite[Theorem~2.2]{AiharaImakuraMorikuni2022SIMAX}. 

This study presented culminating results on CIRS for reducing the residual gap. 
Nevertheless, we will continue the quest for further advances from both theoretical and experimental viewpoints in future studies. 
In the future, we could consider preconditioned and/or complex arithmetic cases and reduce the computational costs by a partial application of the orthonormalization for Bl-CIRS (cf.~\cite[section~7.5]{AiharaImakuraMorikuni2022SIMAX}).

\appendix
\section{Towards general case of orthonormalizations}\label{Appendix} 

Beyond using a particular process for the orthonormalization strategy as discussed in \cref{sec4}, we consider the residual gap in a more general case. 
The quantity~$\| \Delta I_{k-1} \|$ below is regarded as the relative distance from a column-orthonormal matrix $Q_{k-1}$ closest to $\widehat{Q}_{k-1}$.

\begin{lemma} \label{lm:local_error_}
	Let $X_k \in \mathbb{F}^{n \times s}$ and $R_k \in \mathbb{F}^{n \times s}$ be the $k$th approximation and residual, respectively, generated by \eqref{eq:recur_form} in finite precision arithmetic.
	Let $\widehat{Q}_{k-1} \in \mathbb{R}^{n \times s}$ be a numerically computed Q-factor of the QR decomposition of $P_{k-1}$ such that $\widehat{Q}_{k-1} = Q_{k-1} (I_s + \Delta I_{k-1})$ for some $Q_{k-1} \in \mathbb{R}^{n \times s}$ whose columns are exactly orthonormal.
	Assume that $m$, $s$, and $\mathbf{u}$ satisfy 
	\begin{align}
	\| \Delta I_{k-1} \| + \gamma_s\| \widehat{Q}_{k-1} \| < 1 \quad \text{and} \quad (m+s) \mathbf{u} < 1. \label{add_asp}
	\end{align}
	Then, the local errors are bounded as 
	\begin{align}
		\| E_{X_k} \| & \leq \frac{\gamma_s \| \widehat{Q}_{k-1} \|}{1 - \gamma_s \| \widehat{Q}_{k-1} \| - \| \Delta I_{k-1} \|} [ \| X_{k-1} \| + (1+\gamma_1) \| X_k \| ] + \gamma_1 \| X_k \|, \label{eq:E_Xk_}  \\
		\| E_{R_k} \| & \leq \gamma_1 \| R_k \| + \frac{\gamma_{m+s} \| A \|\| \widehat{Q}_{k-1} \|  }{1 - \gamma_s \| \widehat{Q}_{k-1} \| - \| \Delta I_{k-1} \|} [ \| X_{k-1} \| + (1+\gamma_1) \| X_k \| ]. \label{eq:E_Rk_}
	\end{align}	
\end{lemma}

\begin{proof}
	As in \eqref{eq:E_Xk'}, we have 
	\begin{align}
		\| E_{X_k}' \| \leq \gamma_1 \| X_k \|, \label{eq:E_Xk'_} 
	\end{align}
	and it follows from \eqref{eq:fl_model3} that $\| E_{X_k}^\square \| \leq \gamma_s \| \widehat{Q}_{k-1} \| \| \alpha_{k-1}^\square \|$.
	Then, from \eqref{eq:recur_E_Xk_sq}, we have
	\begin{align*}
		\| \widehat{Q}_{k-1} \alpha_{k-1}^\square \| 
		& \leq \| X_{k-1} \| + \| X_k \| + \| E_{X_k}' \| + \| E_{X_k}^\square \| \\
		& \leq \| X_{k-1} \| + (1+\gamma_1) \| X_k \| + \gamma_s \| \widehat{Q}_{k-1} \| \| \alpha_{k-1}^\square \|.
	\end{align*}
	Plugging this into
	\begin{align*}
		\| \alpha_{k-1}^\square \| & = \| Q_{k-1} \alpha_{k-1}^\square \| = \| (\widehat{Q}_{k-1} - Q_{k-1} \Delta I_{k-1}) \alpha_{k-1}^\square \| \\
		& \leq \| \widehat{Q}_{k-1} \alpha_{k-1}^\square \| + \| \Delta I_{k-1} \| \| \alpha_{k-1}^\square \|, 
	\end{align*}
	it holds that
	\begin{align*}
		\| \alpha_{k-1}^\square \| 
		& \leq \| X_{k-1} \| + (1+\gamma_1) \| X_k \| + (\gamma_s \| \widehat{Q}_{k-1} \| + \| \Delta I_{k-1} \|) \| \alpha_{k-1}^\square \|.
	\end{align*}
	We therefore obtain the bound for $\alpha_{k-1}^\square$ as
	\begin{align}
		\| \alpha_{k-1}^\square \| \leq \frac{1}{1 - \gamma_s \| \widehat{Q}_{k-1} \| - \| \Delta I_{k-1} \|} \left[ \| X_{k-1} \| + (1+\gamma_1) \| X_k \| \right], \label{add_alpha}
	\end{align}
	where the denominator is positive from the first assumption in \eqref{add_asp}. 
	Hence, we have
	\begin{align}
		\| E_{X_k}^\square \| 
		& \leq \frac{\gamma_s \| \widehat{Q}_{k-1} \|}{1 - \gamma_s \| \widehat{Q}_{k-1} \| - \| \Delta I_{k-1} \|} \left[ \| X_{k-1} \| + (1+\gamma_1) \| X_k \| \right]. \label{eq:E_Xk_sq_}
	\end{align}
	Summing up \eqref{eq:E_Xk'_} and \eqref{eq:E_Xk_sq_}, we obtain the bound~\eqref{eq:E_Xk_}.
	
	Next, from $\| E_{R_k}' \| \leq \gamma_1 \| R_k \|$, \eqref{eq:fl_model3}, \eqref{eq:fl_model4}, and \eqref{add_alpha}, we have
	\begin{align*}
		\| E_{R_k}^\square \| 
		& \leq \gamma_s \| \mathrm{fl}(A \widehat{Q}_{k-1}) \| \| \alpha_{k-1}^\square \| \\
		& \leq \frac{\gamma_s}{1 - \gamma_s \| \widehat{Q}_{k-1} \| - \| \Delta I_{k-1} \|} (\| A \widehat{Q}_{k-1} \| + \gamma_m \| A \| \| \widehat{Q}_{k-1} \|) \left[ \| X_{k-1} \| + (1+\gamma_1) \| X_k \| \right] \\
		& \leq \frac{\gamma_s (1+\gamma_m) \| A \| \| \widehat{Q}_{k-1} \|}{1 - \gamma_s \| \widehat{Q}_{k-1} \| - \| \Delta I_{k-1} \|} \left[ \| X_{k-1} \| + (1+\gamma_1) \| X_k \| \right].
	\end{align*}
	From \eqref{eq:fl_model4}, $\| E_{R_k}'''\alpha_{k-1}^\square \| \leq \gamma_m \| A \| \| \widehat{Q}_{k-1} \| \|\alpha_{k-1}^\square \|$ also holds. 
	Hence, we have
	\begin{align*}
		\| E_{R_k} \| 
		& \leq \gamma_1 \| R_k \| + \frac{\gamma_s (1 + \gamma_m) \| A \| \| \widehat{Q}_{k-1} \|}{1 - \gamma_s \| \widehat{Q}_{k-1} \| - \| \Delta I_{k-1} \|} [ \| X_{k-1} \| + (1+\gamma_1) \| X_k \| ] \\
		& \qquad + \frac{\gamma_m \| A \| \| \widehat{Q}_{k-1} \|}{1 - \gamma_s \| \widehat{Q}_{k-1} \| - \| \Delta I_{k-1} \|} [ \| X_{k-1} \| + (1+\gamma_1) \| X_k \| ] \\
		& \leq \gamma_1 \| R_k \| + \frac{\gamma_{m+s} \| A \| \| \widehat{Q}_{k-1} \|}{1 - \gamma_s \| \widehat{Q}_{k-1} \| - \| \Delta I_{k-1} \|} [ \| X_{k-1} \| + (1+\gamma_1) \| X_k \| ].
	\end{align*}
	Here, the inequality $\gamma_s(1+\gamma_m) + \gamma_m \leq \gamma_{m+s}$ and the second assumption in \eqref{add_asp} are used. 
\end{proof}

Finally, we evaluate the residual gap by using \Cref{lm:local_error_}.
The following result generalizes those for Givens and Householder orthonormalizations in \cref{sec4} by retaining $\|\widehat{Q}_{k-1}\|$, and thus other orthonormalization results such as from the Gram--Schmidt process are expected to be incorporated. 

\begin{theorem}
	Under the assumptions in \Cref{lm:local_error_} with $X_0 = O$ and 
	\begin{align*}
	\| \Delta I_i \| + \gamma_s\| \widehat{Q}_i \| < 1 \quad \text{for} \quad i<k, \quad \text{and} \quad (m+2s) \mathbf{u} < 1,
	\end{align*}
	the norm of the residual gap~$G_{R_k} := (B-AX_k) - R_k$ is bounded as follows: 
	\begin{align}
\| G_{R_k} \| < \sum_{i=1}^k \frac{2 \gamma_{m+2s} \| A \| \| \widehat{Q}_{i-1} \|}{1 - \gamma_s \| \widehat{Q}_{i-1} \| - \| \Delta I_{i-1} \|} \left( \| X_{i-1} \| + \| X_i \| \right) + \gamma_1 \| A \| \sum_{i=1}^k \| X_i \| + \gamma_1 \sum_{i=1}^k \| R_i \|. \label{eq:G_Rk_}
	\end{align}
\end{theorem}

\begin{proof}
	From \Cref{lm:local_error_}, we have that
	\begin{align*}
		\| G_{R_k} \| 
		& \leq \| A \| \sum_{i=1}^k \| E_{X_i} \| + \sum_{i=1}^k \| E_{R_k} \| \\
		& \leq \sum_{i=1}^k \frac{\gamma_s \| A \| \| \widehat{Q}_{i-1} \|}{1 - \gamma_s \| \widehat{Q}_{i-1} \| - \| \Delta I_{i-1} \|} [ \| X_{i-1} \| + (1+\gamma_1) \| X_i \| ] + \gamma_1 \| A \| \sum_{i=1}^k \| X_i \| \\
		& \qquad + \gamma_1 \sum_{i=1}^k \| R_i \| + \sum_{i=1}^k \frac{\gamma_{m+s} \| A \| \| \widehat{Q}_{i-1} \|}{1 - \gamma_s \| \widehat{Q}_{i-1} \| - \| \Delta I_{i-1} \|} [ \| X_{i-1} \| + (1+\gamma_1) \| X_i \| ] \\
		& < \sum_{i=1}^k \frac{2\gamma_{m+2s} \| A \| \| \widehat{Q}_{i-1} \|}{1 - \gamma_s \| \widehat{Q}_{i-1} \| - \| \Delta I_{i-1} \|} \left( \| X_{i-1} \| + \| X_i \| \right) + \gamma_1 \| A \| \sum_{i=1}^k \| X_i \| + \gamma_1 \sum_{i=1}^k \| R_i \|,
	\end{align*}
where $(\gamma_s + \gamma_{m+s})(1+\gamma_1) < 2(\gamma_s + \gamma_{m+s}) < 2\gamma_{m+2s}$ is used.
\end{proof}

It is an important observation in \eqref{eq:G_Rk_} that the distance $\| \Delta I_{i-1} \|$ is in the denominator, and thus the upper bound of the residual gap can be large if $\| \Delta I_{i-1} \|$ becomes large. 
This also implies that, when using the recursion formulas \eqref{Rec_X_R_P} without the orthonormalization (as an extreme case), the residual gap may increase even if smooth convergence behaviors of the residual and approximation norms are obtained; such a phenomenon is often observed in numerical experiments (e.g., see \Cref{Ex1_3}).

\end{document}